\documentclass[11pt,a4paper]{amsart}

\usepackage{amsmath,amssymb,amsthm,mathtools}
\usepackage{enumitem}
\usepackage{geometry}
\usepackage{microtype}
\usepackage[colorlinks=true,linkcolor=blue,citecolor=blue,urlcolor=blue]{hyperref}
\usepackage[nameinlink,capitalize,noabbrev]{cleveref}
\usepackage{tikz-cd}
\usepackage{etoolbox}
\patchcmd{\abstract}
{\leftmargin3pc}
{\leftmargin0pt}
{}{}  

\numberwithin{equation}{section}
\allowdisplaybreaks
\setlist[enumerate,1]{label=(\arabic*),leftmargin=2.3em}

\newtheorem{theorem}{Theorem}[section]
\newtheorem{proposition}[theorem]{Proposition}
\newtheorem{lemma}[theorem]{Lemma}
\newtheorem{corollary}[theorem]{Corollary}
\theoremstyle{definition}
\newtheorem{definition}[theorem]{Definition}

\theoremstyle{remark}
\newtheorem{remark}[theorem]{Remark}
\newtheorem{maintheorem}{Theorem}

\title[Cartan Subalgebras and Inductive Limits ]
{Cartan Subalgebras of Dimension Drop Interval Algebras and Their Inductive Limits}
\author{Xinlong Cui}
\address{School of Mathematics and Statistics, Northeast Normal University, Changchun, {\rm130024}, China}
\email{xlcui@nenu.edu.cn}
\date{}
\hypersetup{
  pdftitle={Cartan Subalgebras of Dimension Drop Interval Algebras and Their Inductive Limits},
  pdfauthor={Xinlong Cui},
  pdfsubject={Cartan Subalgebras of Dimension Drop Interval Algebras and Their Inductive Limits},
  pdfkeywords={ Cartan subalgebra; C*-diagonal; dimension drop interval algebra; inductive limit; Weyl groupoid; dense amalgam.
}
}

\begin{document}

\begin{abstract}
We describe the conjugacy of non-degenerate Cartan subalgebras of the dimension drop interval algebras $M_m(\widetilde{\mathbb I}_p)$ in terms of double cosets of permutation groups, and describe their spectra by finite bi-partite multigraphs. All Cartan subalgebras are non-degenerate precisely when the two parameters are coprime. We also give splitting interval examples with prescribed finite numbers of Cartan conjugacy classes. Passing to inductive limits, we construct a unital simple AD algebra $\mathcal{A}$ of real rank zero and stable rank one, with $K_0(\mathcal{A})\cong\mathbb Z[1/3]$ and $K_1(\mathcal{A})\cong\mathbb Z/2\mathbb Z$. It contains infinitely many mutually isomorphic but pairwise nonconjugate $\mathrm{C}^*$-diagonals. Their spectra are all homeomorphic to the dense amalgam of circles, while orbit multiplicities in the Weyl groupoids distinguish their embeddings. The same AD algebra also contains infinitely many pairwise nonisomorphic Cartan subalgebras.
	
\par\vspace{0.8\baselineskip}
\end{abstract}

\maketitle

\section{Introduction}

Cartan subalgebras connect operator algebras with the spaces and dynamical systems from which they arise. Their origins lie in Dixmier's study of maximal abelian subalgebras of finite factors \cite{D2} and Singer's work on automorphisms \cite{S}. Feldman and Moore subsequently described von Neumann algebras with Cartan subalgebras through measured equivalence relations and cocycles \cite{FM,FM2}. The corresponding topological theory developed through Renault's groupoid construction \cite{R2}, Kumjian's $\mathrm{C}^*$-diagonals \cite{K}, and Renault's reconstruction theorem for $\mathrm{C}^*$-algebraic Cartan pairs \cite{R3}.

A Cartan subalgebra of a $\mathrm{C}^*$-algebra is a maximal abelian subalgebra which contains an approximate identity, is regular, and admits a faithful conditional expectation. A $\mathrm{C}^*$-diagonal additionally has the unique extension property for pure states \cite{A,K}. For a separable Cartan pair, the inclusion determines a twisted, topologically principal, \'etale groupoid; the groupoid is principal in the diagonal case \cite{K,R3}. This correspondence makes the inclusion itself an object of study: the spectrum records the unit space, whereas the normalizers also encode the arrows and their action. Consequently, isomorphism of the abelian subalgebras and conjugacy inside the ambient algebra are distinct classification problems. Throughout, conjugacy means conjugacy by an automorphism of the ambient algebra.

Concrete models illustrate the scope of this viewpoint. Finite-dimensional approximations and Bratteli diagrams underlie the canonical diagonals of AF algebras \cite{B,R2}, while directed graphs provide another source of groupoid $\mathrm{C}^*$-algebras \cite{KPRR}. For Cantor minimal systems, orbit equivalence interacts with ordered $K$-theory and crossed products \cite{GPS}. More generally, continuous orbit equivalence of topologically free actions can be expressed through isomorphisms preserving the canonical Cartan subalgebras \cite{L}. These examples motivate both the search for Cartan subalgebras and the classification of their embeddings.

There is also a close connection with the structure of nuclear $\mathrm{C}^*$-algebras. The Universal Coefficient Theorem of Rosenberg and
Schochet \cite{RS} holds for amenable groupoid $\mathrm{C}^*$-algebras by work of Tu \cite{T}. Barlak and Li extended this connection to Cartan pairs and obtained reformulations of the UCT problem \cite{BL,BL2}. Alongside the advances in classification, quasidiagonality, and nuclear dimension
\cite{C,EGLN,TWW}, Li proved that every classifiable simple $\mathrm{C}^*$-algebra has a Cartan subalgebra \cite{L2}. Thus existence is established throughout the classifiable setting, while the classification of the resulting pairs remains a finer problem than classification of the ambient algebras.

Dimension drop interval algebras provide a natural setting for an explicit study of this problem. They consist of matrix-valued functions on an interval with prescribed smaller matrix algebras at the endpoints, and belong to the class of one-dimensional noncommutative CW complexes \cite{ELP}. In the usual algebra $I_{r,s}$, the interior fibre is $M_r\otimes M_s$, with endpoint fibres $M_r\otimes1_s$ and $1_r\otimes M_s$. When $r$ and $s$ are coprime, these are the prime dimension-drop building blocks used to construct the Jiang--Su algebra
$\mathcal Z$ \cite{JS}. Tensorial absorption of $\mathcal Z$ is central to regularity and classification; see \cite{C,RW,TW}.

The related dimension drop blocks with scalar values at both endpoints occur in AD algebras. Their inductive limits and the role of $K$-theory with coefficients were studied by Dadarlat, Dadarlat--Loring, and Eilers--Toms \cite{D,DL,ET}. One-dimensional noncommutative CW complexes also admit classification results using the Cuntz semigroup, under the relevant $K_1$-vanishing assumptions \cite{R4}. From the viewpoint of Cartan inclusions, these building blocks retain a particularly accessible geometry: diagonal strands in the interior are identified at the endpoints. The boundary identifications therefore influence both the spectrum and the allowable conjugacies.

Li and Renault analysed Cartan subalgebras of homogeneous algebras using principal bundles and established several nonuniqueness results \cite{LR2}. Barlak and Raum treated non-degenerate Cartan subalgebras of dimension drop algebras, relating conjugacy to finite combinatorial data and spectral topology \cite{BR}. At the level
of simple algebras, the dynamical realization of $\mathcal Z$ by Deeley--Putnam--Strung \cite{DPS} supplies another source of Cartan models. For classifiable stably finite unital algebras with torsion-free $K_0$ and trivial $K_1$, Li constructed continuum many nonconjugate C\textsuperscript{*}-diagonals with Menger-curve spectra \cite{L3}. Mitscher and Spielberg produced Cartan subalgebras in AF algebras that are isomorphic but not conjugate to the standard diagonals \cite{MS}.

Passing from building blocks to limits requires control of the connecting maps as well as the Cartan subalgebras. Inductive limit constructions for Cartan pairs were developed in \cite{BL2,L2}, and Austin--Mitra constructed groupoid models for entire inductive systems \cite{AM}. Li and Raad obtained $\mathrm{C}^*$-diagonals in broad classes of AH algebras, including examples beyond the classifiable setting
\cite{LR}. Raad further described spectral components and studied uniqueness for generalized diagonal connecting maps \cite{R}. These developments motivate an explicit analysis of which finite stage features survive in the limit and which distinguish embeddings with the same spectrum.

We study these questions for  dimension drop interval algebras
\[
M_m(\widetilde{\mathbb I}_p)
=\{f\in C([0,1],M_p\otimes M_m):
f(0),f(1)\in1_p\otimes M_m\}.
\]
A Cartan subalgebra is non-degenerate if both endpoint fibres have
dimension $m$. Put $\underline q=\{1,\ldots,q\}$ and write
$K_{p,m}=\operatorname{Sym}(p)\times\operatorname{Sym}(m)$, acting coordinatewise on
$\underline p\times\underline m$. Our first result gives a permutation model and its
conjugacy relation.

\begin{maintheorem}[See Theorems~\ref{biaozhunxing}, \ref{fenlei}, and~\ref{feituihua}]
Every non-degenerate Cartan subalgebra of $M_m(\widetilde{\mathbb I}_p)$ is conjugate to a twisted standard Cartan subalgebra $B_\sigma$, where $\sigma\in\operatorname{Sym}(\underline p\times\underline m)$. Moreover,
\[
B_\sigma\text{ and }B_\pi\text{ are conjugate}
\quad\Longleftrightarrow\quad
\sigma\in K_{p,m}\pi K_{p,m}\,\cup\,K_{p,m}\pi^{-1}K_{p,m}.
\]
Every Cartan subalgebra of $M_m(\widetilde{\mathbb I}_p)$ is non-degenerate if and only if $\gcd(m,p)=1$.
\end{maintheorem}

The endpoint permutations must be compatible with the full endpoint matrix
algebras. This requirement determines the subgroup $K_{p,m}$; inversion
accounts for reversal of the interval. Spectra have a simpler description
through the reduced matrix, which counts how the diagonal strands join
the endpoint characters. We next consider Cartan subalgebras up to abstract isomorphism.

\begin{samepage}
\begin{maintheorem}[See Proposition~\ref{m2}]
For every $n\geq1$, the  subhomogeneous $\mathrm{C}^*$-algebra $M_n(\widetilde{\mathbb I}_2)$ has exactly $n$ non-degenerate Cartan subalgebras up to isomorphism as abelian $\mathrm{C}^*$-algebras.

\end{maintheorem}
\end{samepage}

However, two Cartan subalgebras may be isomorphic without being conjugate. For example, $M_2(\widetilde{\mathbb I}_2)$ has five conjugacy classes of non-degenerate Cartan subalgebras but only two isomorphism classes; see Proposition~\ref{wuge} and Remark~\ref{liangzhong}. We also revisit the finite enumeration question considered in \cite{BR} and give an explicit construction using splitting interval algebras. For $1\leq a<b$, the splitting interval algebra with interior fibre $M_{a+b}$ and endpoint algebras $M_a\oplus M_b$ has exactly $a+1$ Cartan conjugacy classes.
\begin{maintheorem}[See Theorem~\ref{zhiding}]
For every $n\geq1$, there exists a subhomogeneous $\mathrm{C}^*$-algebra with exactly $n$ Cartan subalgebras up to conjugacy. All of its Cartan subalgebras are mutually isomorphic and are $\mathrm{C}^*$-diagonals.
\end{maintheorem}

Our inductive limit results concern a fixed unital AD algebra $\mathcal{A}$. We construct Cartan subalgebras in $\mathcal{A}$ whose spectra are inverse limits of finite unions of circles. We identify the resulting space with the dense amalgam of
circles, denoted by $\widetilde{\sqcup}(S^1)$, in the sense of \'Swi\k{a}tkowski \cite{S2}.

\begin{maintheorem}[See Corollary~\ref{wuxian} and Theorem~\ref{michou}]
The algebra $\mathcal{A}$ contains infinitely many mutually isomorphic but pairwise nonconjugate Cartan subalgebras, all of which are $\mathrm{C}^*$-diagonals and have spectrum homeomorphic to $\widetilde{\sqcup}(S^1)$.
\end{maintheorem}

The algebra $\mathcal{A}$ satisfies $K_1(\mathcal{A})\cong\mathbb Z/2\mathbb Z$.
Thus it lies outside the scope of \cite[Theorem~1.4]{L3}, which assumes torsion-free $K_0$ and trivial $K_1$. Here the common spectrum is the dense amalgam of circles. To distinguish the resulting inclusions, we use the numbers of
points at which Weyl groupoid orbits meet persistent circle components. The resulting conjugacy invariant satisfies $\mathfrak M(B^{[r]})=\{2,2r\}$ for $r\geq1$.

Finally, the $\mathcal Z$-stability of $\mathcal{A}$ allows us to apply the tensor product method of \cite{LR2} and vary the covering dimensions of Cartan spectra.

\begin{maintheorem}[See Corollary~\ref{butonggou}]
The same algebra $\mathcal{A}{}$ contains infinitely many Cartan subalgebras that
 are pairwise nonisomorphic as abelian $\mathrm{C}^*$-algebras.
\end{maintheorem}

Section~\ref{yubei} recalls the preliminaries. Section~\ref{youxian} proves the finite stage classification and spectral results. Section~4 treats prescribed finite numbers of conjugacy classes, and Section~\ref{guina} develops the inductive-limit constructions.

\section{Preliminaries}\label{yubei}

Throughout, $\mathbb{N}$, $\mathbb{Z}$, $\mathbb C$, and $\mathbb{T}$ denote the positive integers, the integers, the complex numbers, and the unit circle.  For a finite set $X$, let $\operatorname{Sym}(X)$ be its permutation group.  We write $\underline n=\{1,\ldots,n\}$, $\operatorname{Sym}(n)=\operatorname{Sym}(\underline n)$. For positive integers $r,s$, let $\operatorname{pr}_2:\underline r\times\underline s\to\underline s$
denote projection onto the second coordinate; thus $\operatorname{pr}_2(i,j)=j$, and identify a permutation with its permutation matrix. The standard matrix units are denoted by $e_{ij}$.

\begin{definition}
Let $B\subseteq A$ be $\mathrm{C}^*$-algebras. The normalizer of $B$ in $A$ is
\[
 N_A(B)=\{x\in A:x^*Bx\subseteq B\text{ and }xBx^*\subseteq B\}.
\]
The subalgebra $B$ is regular if $N_A(B)$ generates $A$ as a $\mathrm{C}^*$-algebra.
\end{definition}

\begin{definition}
A linear map $E:A\to B$ is a faithful conditional expectation if it is completely positive, $E|_B=\operatorname{id}_B$, $E(b_1ab_2)=b_1E(a)b_2$, and $E(a^*a)=0$ implies $a=0$.
\end{definition}

\begin{definition}[{\cite[Definition~5.1]{R3}}]
A $\mathrm{C}^*$-subalgebra $B\subseteq A$ is a Cartan subalgebra if
\begin{enumerate}
 \item $B$ is maximal abelian;
 \item $B$ contains an approximate identity for $A$;
 \item there is a faithful conditional expectation $A\to B$;
 \item $B$ is regular in $A$.
\end{enumerate}
The pair $(A,B)$ is then called a Cartan pair. Two Cartan subalgebras $B_1,B_2\subseteq A$ are conjugate if some $\alpha\in\operatorname{Aut}(A)$ satisfies $\alpha(B_1)=B_2$.
\end{definition}
Let $A$ be a commutative $\mathrm{C}^*$-algebra. Then its unique Cartan subalgebra is $A$ itself. The diagonal $D_q\subseteq M_q$ is a Cartan subalgebra. Its normalizer
consists of the matrices with at most one nonzero entry in each row and each column.

\begin{definition}
For $m,p\in\mathbb{N}$, set
\begin{equation*}
 M_m(\tilde{\mathbb{I}}_p)
 =\{f\in C([0,1],M_p\otimes M_m): f(0),f(1)\in1_p\otimes M_m\}.
\end{equation*}
A Cartan subalgebra $B\subseteq M_m(\tilde{\mathbb{I}}_p)$ is called non-degenerate if $\dim B_0=\dim B_1=m$.
\end{definition}

The centre of $M_m(\tilde{\mathbb{I}}_p)$ is $C([0,1])$. Hence an automorphism of $M_m(\tilde{\mathbb{I}}_p)$ induces a homeomorphism of $[0,1]$. Such a homeomorphism either fixes or interchanges the endpoints. Next we need a classification result for the special case of homogeneous $\mathrm{C}^*$-algebras over an interval.

\begin{theorem}[{\cite{LR2}}]\label{qici}
Let $I$ be an interval and let $A=C_0(I,M_m)$. Then $A$ admits a unique Cartan subalgebra up to conjugacy by an inner automorphism implemented by a unitary in its multiplier algebra $M(A)=C_b(I,M_m)$.
\end{theorem}

\section{Cartan subalgebras of dimension drop interval algebras} \label{youxian}

\subsection{Non-degenerate Cartan subalgebras and their conjugacy classes}
For convenience, we first introduce the one-sided dimension drop algebra; see \cite{BR}.

\begin{definition}
The one-sided dimension drop algebra for $m, n \geqslant1 $ is
\[ {J}_{m,n}=\{f\in C([0,1],M_m\otimes M_n):f(0)\in M_m\otimes 1\}.\]
\end{definition}

We call the subalgebra\[C_{m,n}=\{f\in C([0,1],D_m\otimes D_n): f(0)\in D_m\otimes 1\}\] its standard Cartan subalgebra.

\begin{lemma}[{\cite[Lemma~3.4]{BR}}]\label{danbian}
For every Cartan subalgebra $B\subseteq J_{1,n}$, there is a unitary $u\in C_b((0,1],U(n))$ such that $(\operatorname{Ad}u)(B)=C_{1,n}.$ In particular, $J_{1,n}$ has a unique Cartan subalgebra up to conjugacy by an automorphism.
\end{lemma}

The following result was stated in \cite[Proposition~3.7]{BR}. We give a self-contained proof, as it will be used in the classification below.

\begin{theorem}\label{danbianweiyi}
The one-sided dimension drop algebras have a unique non-degenerate Cartan subalgebra up to conjugacy.
\end{theorem}

\begin{proof}
For convenience, we write $A=J_{m,n}$, $C=C_{m,n}$. Let $B\subseteq J_{m,n}$ be a non-degenerate Cartan subalgebra. After conjugating by a constant unitary, we may assume that $B(0)=D_m\otimes 1_n$. Put $p_i=e_{ii}\otimes 1_n$. By semiprojectivity of $\mathbb C^m$, after restricting to a sufficiently small interval $[0,\varepsilon]$, there are orthogonal projections $f_1,\ldots,f_m\in B_{[0,\varepsilon]}$ such that
\[f_i(0)=p_i,\qquad \sum_{i=1}^m f_i=1,\qquad \|f_i-p_i\|<1.\]
Then the polar decomposition of $p_if_i(t)$ provides us with partial isometries $v_i(t)$ such that
\[v_i(t)^*v_i(t)=f_i(t),\qquad v_i(t)v_i(t)^*=p_i,\qquad v_i(0)=p_i.\]
Let $v=v_1+\cdots+v_m$. Then $v$ is a unitary and $v f_i v^*=p_i$, $i=1,\ldots,m$. Since $v(0)=1$, we may extend $v$ to a unitary $\widetilde v\in A$ which agrees with $v$ on $[0,\varepsilon]$ and is equal to $1$ away from $0$. After conjugating $B$ by $\widetilde v$, we may therefore assume that $p_i\in B_{[0,\varepsilon]}$.

Put $s_i=e_{i1}\otimes 1_n$. Regularity gives $p_iA_{[0,\varepsilon]}p_1
=\overline{\operatorname{span}}\bigl(p_iN_{A_{[0,\varepsilon]}}(B_{[0,\varepsilon]})p_1\bigr).$ Since $p_iA_0p_1=\mathbb C s_i$, regularity allows us to choose
$x_i\in p_iN_{A_{[0,\varepsilon]}}(B_{[0,\varepsilon]})p_1$, $x_i(0)=s_i$. Since $x_i^*x_i\in p_1B_{[0,\varepsilon]}p_1$ and $(x_i^*x_i)(0)=p_1$, after shrinking $\varepsilon$ if necessary, $x_i^*x_i$ is invertible in the corner
$p_1A_{[0,\varepsilon]}p_1$. Hence $n_i=x_i(x_i^*x_i)^{-1/2}$ is a normalising partial isometry satisfying
\[n_i^*n_i=p_1,\qquad n_in_i^*=p_i,\qquad n_i(0)=s_i.\]
Consequently, there is a continuous map $z_i\colon[0,\varepsilon]\to U(n)$ such that
\[
n_i(t)=e_{i1}\otimes z_i(t),
\qquad z_i(0)=1_n.
\]
Since $p_1B_{[0,\varepsilon]}p_1$ is a Cartan subalgebra of $p_1A_{[0,\varepsilon]}p_1\cong J_{1,n}$, Lemma~\ref{danbian} provides $q\in C_b((0,\varepsilon],U(n))$ such that, with $r_1(t)=e_{11}\otimes q(t)$, $r_1\bigl(p_1B_{[0,\varepsilon]}p_1\bigr)r_1^* =p_1C_{[0,\varepsilon]}p_1$. For $i=1,\ldots,m$, set
\[
U_i(t)=s_i r_1(t)n_i(t)^*,
\qquad
U(t)=\sum_{i=1}^m U_i(t).
\]
For each $t\in(0,\varepsilon]$, we have
$U_i(t)^*U_i(t)=U_i(t)U_i(t)^*=p_i$.
The $p_i$ are mutually orthogonal and sum to $1$; hence $U(t)=\sum_iU_i(t)$ is unitary. Moreover, $U(t)=\sum_{i=1}^m e_{ii}\otimes q(t)z_i(t)^*$. Thus, for $a=(a_{ij})\in M_m$,
\[
U(t)(a\otimes1_n)U(t)^*
=\sum_{i,j=1}^m a_{ij}e_{ij}\otimes
q(t)z_i(t)^*z_j(t)q(t)^*.
\]
As $z_i(t)\to1_n$, we have
\[
\|q(t)z_i(t)^*z_j(t)q(t)^*-1_n\|
=\|z_i(t)^*z_j(t)-1_n\|\longrightarrow0.
\]
The analogous limit for conjugation by $U(t)^*$ shows that $\operatorname{Ad}U$ extends to an automorphism of $A_{[0,\varepsilon]}$ fixing the endpoint quotient. The normalising property of $n_i$ gives
\[
n_i^*(p_iB_{[0,\varepsilon]}p_i)n_i
=p_1B_{[0,\varepsilon]}p_1.
\]
Lemma~\ref{danbian} and the definition of $U_i$ therefore imply
\[
(\operatorname{Ad}U)(p_iB_{[0,\varepsilon]}p_i)
=p_iC_{[0,\varepsilon]}p_i.
\]
As the $p_i$ belong to $B_{[0,\varepsilon]}$ and sum to $1$, we obtain
$(\operatorname{Ad}U)(B_{[0,\varepsilon]}) =C_{[0,\varepsilon]}$. Finally, using the path connectedness of $U(mn)$, extend $\operatorname{Ad}U$ to an automorphism of $A$ which is the identity away from $0$. After conjugating $B$ by this automorphism, we may assume that $B_{[0,\varepsilon]}=C_{[0,\varepsilon]}$.

Choose $0<\delta<\varepsilon$. By \cref{qici}, there exists $w\in C([\delta,1],U(mn))$ such that $wB_{[\delta,1]}w^*=C_{[\delta,1]}$. Since $B_{[\delta,\varepsilon]}=C_{[\delta,\varepsilon]}$, we have $w(t)\in N_{U(mn)}(D_m\otimes D_n)$, $t\in[\delta,\varepsilon]$. Define
\[
y(t)=
\begin{cases}
w(t)^*,& t\in[\delta,\varepsilon],\\
w(\varepsilon)^*,& t\in[\varepsilon,1],
\end{cases}
\qquad
w'=yw.
\]
Then $y$ normalises $C_{[\delta,1]}$, so
\[
w'B_{[\delta,1]}(w')^*=C_{[\delta,1]},
\qquad
w'|_{[\delta,\varepsilon]}=1.
\]
Extending $w'$ by $1$ on $[0,\delta]$ gives a unitary $\widetilde w\in A$ satisfying $\widetilde wB\widetilde w^*=C$. Composing this inner automorphism with the preceding conjugacies completes the proof.
\end{proof}

\begin{proposition}\label{biaozhun}
The algebra
\[
B_{\mathrm{st}} =M_m(\tilde{\mathbb{I}}_p)\cap C([0,1],D_p\otimes D_m)
\]
is a non-degenerate Cartan subalgebra of $M_m(\tilde{\mathbb{I}}_p)$.
\end{proposition}

\begin{proof}
The canonical diagonal expectation restricts to a faithful conditional expectation $E\colon M_m(\tilde{\mathbb{I}}_p)\longrightarrow B_{\mathrm{st}}$. For every $t\in[0,1]$, the fibre $B_{\mathrm{st}}(t)$ is a MASA of $M_m(\tilde{\mathbb{I}}_p)(t)$. Hence $B_{\mathrm{st}}$ is maximal abelian in $M_m(\tilde{\mathbb{I}}_p)$.
	
Regularity may be checked locally. Over an interior interval, the inclusion is isomorphic to
\[
C(\overline V,D_p\otimes D_m) \subseteq C(\overline V,M_p\otimes M_m),
\]
which is regular. Near either endpoint, after interchanging the tensor factors, it is isomorphic to the standard Cartan inclusion $C_{m,p}\subseteq J_{m,p}$, which is also regular. Thus $B_{\mathrm{st}}$ is Cartan. Finally,
\[
B_{\mathrm{st}}(0)=B_{\mathrm{st}}(1) =1_p\otimes D_m,
\]
so it is non-degenerate.
\end{proof}

\begin{definition}\label{niuzhuan}
Let $\sigma\in\operatorname{Sym}(\underline{p}\times \underline{m})$, viewed as a permutation unitary in $M_p\otimes M_m$.  Choose a continuous unitary path
\[
u^\sigma\in C([1/3,2/3],M_p\otimes M_m),
\qquad u^\sigma_{1/3}=\sigma,
\quad u^\sigma_{2/3}=\mathbf{1}.
\]
Extend it by $u_t^\sigma=\sigma$ on $[0,1/3]$ and $u_t^\sigma=1$ on $[2/3,1]$, and define
\begin{equation*}
 B_\sigma
 =u^\sigma C([0,1],D_p\otimes D_m)(u^\sigma)^*\cap M_m(\tilde{\mathbb{I}}_p).
\end{equation*}
This is a non-degenerate Cartan subalgebra, called a twisted standard
Cartan subalgebra.
\end{definition}

It is straightforward to verify that, up to conjugacy in $M_m(\tilde{\mathbb{I}}_p)$, the resulting Cartan subalgebra is independent of
both the choice of the interpolating unitary and the choice of the
subdivision points. Conversely, it turns out that up to conjugacy, every non-degenerate Cartan subalgebra of $M_m(\tilde{\mathbb{I}}_p)$ is of this form.

\begin{theorem}\label{biaozhunxing}
Every non-degenerate Cartan subalgebra of $M_m(\tilde{\mathbb{I}}_p)$ is conjugate to $B_\sigma$ for some $\sigma\in\operatorname{Sym}(\underline{p}\times \underline{m})$.
\end{theorem}

\begin{proof}
Let $B\subseteq M_m(\tilde{\mathbb{I}}_p)$ be non-degenerate. The restrictions of $M_m(\tilde{\mathbb{I}}_p)$ to small neighborhoods of the endpoints are one-sided dimension drop algebras $J_{m,p}$. By \cref{danbianweiyi}, endpoint automorphisms carry the two restrictions of $B$ to the standard Cartan subalgebras. Extend these automorphisms toward the interior, making each one the identity near the opposite endpoint.  After conjugating, we may suppose that $B$ agrees with $B_{\mathrm{st}}$ on $[0,1/3]\cup[2/3,1]$.
	
On $[1/3,2/3]$, \cref{qici} gives a unitary path $w$ such that
\[ wB|_{[1/3,2/3]}w^*=C([1/3,2/3],D_p\otimes D_m). \]
At each endpoint, $w$ is a monomial unitary. Multiplication by a continuous monomial path does not change the diagonal algebra.  We may therefore arrange that $w_{2/3}=1$ and $w_{1/3}=\sigma$ is a permutation unitary. The resulting Cartan algebra is $B_{\sigma^{-1}}$.
\end{proof}

\begin{lemma}\label{zhongxin}
Let $\varphi$ be an orientation-preserving homeomorphism of $[0,1]$. For every $\sigma$, there is
$\alpha\in\operatorname{Aut}(B_\sigma\subseteq M_m(\tilde{\mathbb{I}}_p))$ whose restriction to the centre is $\varphi^*$.
\end{lemma}

\begin{proof}
With $u=u^\sigma$, put $v_t=u_{\varphi(t)}$. Since $(uv^*)_0=(uv^*)_1=1$, the formula
$\alpha=\operatorname{Ad}(uv^*)\circ\varphi^*$ defines the required automorphism.
\end{proof}

In studying conjugacy of the Cartan subalgebras $B_\sigma$, it is useful to distinguish automorphisms according to their action on the base interval. Note that the centre of $M_m(\tilde{\mathbb{I}}_p)$ is isomorphic to $C([0,1])$. Hence every automorphism $\alpha$ of $M_m(\tilde{\mathbb{I}}_p)$ induces a homeomorphism of $[0,1]$. We call $\alpha$ orientation preserving if the induced homeomorphism is orientation preserving. Let $\Phi(t)=1-t$.  Precomposition by $\Phi$ is an automorphism of $M_m(\tilde{\mathbb I}_p)$. Reversing an implementing path shows that $\Phi^*(B_\sigma)$ is conjugate to $B_{\sigma^{-1}}$. For convenience, let $K_{p,m}=\operatorname{Sym}(p)\times\operatorname{Sym}(m)$, acting coordinatewise on $\underline p\times\underline m$ by $(a,b)(i,j)=(a(i),b(j))$.

\begin{theorem} \label{fenlei}
Let $\sigma,\pi\in\operatorname{Sym}(\underline{p}\times \underline{m})$.  Then $B_\sigma$ and $B_\pi$ are conjugate in $M_m(\tilde{\mathbb{I}}_p)$ if and only if
\begin{equation*}
\sigma\in K_{p,m}\pi K_{p,m}
\quad\text{or}\quad
\sigma\in K_{p,m}\pi^{-1}K_{p,m}.
\end{equation*}
\end{theorem}

\begin{proof}
First consider the case in which the automorphism is orientation preserving. Suppose that $\rho_1,\rho_2\in K_{p,m}$. Choose a continuous unitary~$v\in C([0,1],M_p\otimes M_m)$~such that $v_t=\rho_1$ on $[0,1/3]$ and $v_t=\rho_2$ on $[2/3,1]$. Since $\rho_1$ and $\rho_2$ normalize the full endpoint algebra $1_p\otimes M_m$, the map $\operatorname{Ad} v$ restricts to an automorphism of $M_m(\tilde{\mathbb{I}}_p)$. Let $u=u^\sigma$ be an implementing path for $B_\sigma$. Conjugation by $v$ preserves the diagonal fibres outside $[1/3,2/3]$. On the middle interval, since $\rho_2$ normalizes $D_p\otimes D_m$, we have
\[v_tu_t(D_p\otimes D_m)u_t^*v_t^*=v_tu_t\rho_2^{-1}(D_p\otimes D_m)\rho_2u_t^*v_t^*.\]
Therefore the relevant implementing path is $w_t=v_tu_t\rho_2^{-1},t\in[1/3,2/3]$.
Its endpoint values are $w_{1/3}=\rho_1\sigma\rho_2^{-1}$, $w_{2/3}=1$. Consequently, $\operatorname{Ad} v(B_\sigma)
 =B_{\rho_1\sigma\rho_2^{-1}}.$
 
Conversely, suppose that $\alpha\in\operatorname{Aut}(M_m(\tilde{\mathbb{I}}_p))$ and $\alpha(B_\sigma)=B_\pi$. By composing with the automorphism from Lemma~\ref{zhongxin}, we may assume that~$\alpha|_{C([0,1])}=\operatorname{id}$. Let $u_\sigma$ and $u_\pi$ be implementing paths for $B_\sigma$ and $B_\pi$, respectively. On $[1/3,2/3]$, the automorphism $\beta=(\operatorname{Ad} u_\pi)^{-1}\circ\alpha\circ\operatorname{Ad} u_\sigma$ preserves~$C([1/3,2/3],D_p\otimes D_m)$. It is therefore implemented fibrewise by a continuous monomial unitary of the form
\begin{equation*}
\tau\,\operatorname{diag}
\bigl(\lambda_{1,1}(t),\ldots,\lambda_{p,m}(t)\bigr),
\end{equation*}
where each $\lambda_{i,j}$ is a continuous $\mathbb{T}$-valued function and
$\tau\in\operatorname{Sym}(\underline p\times\underline m)$.  Since the interval is connected and the permutation group is discrete, $\tau$ is independent of $t$. For a diagonal-preserving fibre automorphism $\gamma$, write $[\gamma]$ for the permutation part of a monomial implementation. On $(0,1/3]$, the automorphism $\alpha$ preserves the standard diagonal, so its permutation part is constant; denote it by~$[\alpha_t]=\rho_0$,~$t\in(0,1/3]$. Similarly, there is $\rho_1\in\operatorname{Sym}(\underline{p}\times \underline{m})$ such that~$[\alpha_t]=\rho_1$,~$t\in[2/3,1)$. We claim that~$\rho_0,\rho_1\in K_{p,m}$.

We prove this for $\rho_0$. Since $\alpha_0$ preserves the endpoint diagonal $1_p\otimes D_m$, there exist~$b\in\operatorname{Sym}(m)$ and $a_j\in\operatorname{Sym}(p)$,~$ j\in\underline m$, such that $\rho_0(i,j)=(a_j(i),b(j))$. On the other hand, $\alpha_0$ is an automorphism of the full endpoint algebra $1_p\otimes M_m$. Hence, for every $j,k\in\underline m$, there is $\mu_{jk}\in\mathbb{T}$ such that $\alpha_0(1_p\otimes e_{jk})=\mu_{jk}(1_p\otimes e_{b(j),b(k)}).$ For $t\in(0,1/3]$, a monomial implementation of $\alpha_t$ gives
\begin{equation*}
\alpha_t(1_p\otimes e_{jk}) = \sum_{i=1}^{p}c_i^{jk}(t) e_{a_j(i),a_k(i)}\otimes e_{b(j),b(k)},
\qquad |c_i^{jk}(t)|=1.
\end{equation*}
As $t\to0$, the right-hand side converges in norm to $\alpha_0(1_p\otimes e_{jk})$. Therefore no off-diagonal matrix unit can occur in the first tensor factor. It follows that
\[a_j(i)=a_k(i)\qquad(i\in\underline p,\ j,k\in\underline m). \]
Thus all the $a_j$ are equal to one permutation $a\in\operatorname{Sym}(p)$, and hence $\rho_0(i,j)=(a(i),b(j))$. Consequently, $\rho_0\in\operatorname{Sym}(p)\times\operatorname{Sym}(m)=K_{p,m}$. The same argument at the right endpoint shows that $\rho_1\in K_{p,m}$.

We now compare the permutation part of $\beta$ at the endpoints of the middle interval. Since $u_\sigma(1/3)=\sigma$, $u_\pi(1/3)=\pi$, we have $[\beta_{1/3}]=\pi^{-1}\rho_0\sigma.$ At $2/3$, both implementing paths are equal to $1$, and hence $[\beta_{2/3}]=\rho_1$. The permutation part of $\beta$ is constant, so $\pi^{-1}\rho_0\sigma=\rho_1.$ Equivalently, $\sigma=\rho_0^{-1}\pi\rho_1 \in K_{p,m}\pi K_{p,m}.$

Finally, suppose that $\alpha$ reverses orientation. By the reversal observation above, there is an orientation-reversing automorphism carrying $B_\pi$ onto $B_{\pi^{-1}}$. Composing it with $\alpha$ gives an orientation-preserving automorphism carrying $B_\sigma$ onto $B_{\pi^{-1}}$. The orientation-preserving case therefore yields $\sigma\in K_{p,m}\pi^{-1}K_{p,m}$.
\end{proof}

\begin{theorem}\label{feituihua}
Every Cartan subalgebra of $M_m(\tilde{\mathbb{I}}_p)$ is non-degenerate if and only if $\gcd(m,p)=1$.
\end{theorem}

\begin{proof}
Restrict a Cartan subalgebra to a sufficiently small interval adjacent to an endpoint and interchange the tensor factors.  This gives a Cartan subalgebra of $J_{m,p}$.  The divisibility argument of \cite[Theorem~4.5]{BR} shows that a degenerate endpoint fibre forces a nontrivial common divisor of $m$ and $p$.
	
Conversely, let $d>1$ divide both $m$ and $p$.  The finite-dimensional orthogonal-MASA construction of \cite[Examples~4.3--4.4]{BR} gives a degenerate Cartan subalgebra in $J_{m,p}$.  Glue it, over a middle interior fibre, to a continuous conjugate of $D_p\otimes D_m$ which becomes the standard endpoint diagonal at the other end.  The resulting subalgebra of $M_m(\tilde{\mathbb{I}}_p)$ is Cartan and remains degenerate at the first endpoint.
\end{proof}

\begin{remark}\label{duijiao}
If $m$ and $p$ are relatively prime, then every Cartan subalgebra of $M_m(\tilde{\mathbb{I}}_p)$ is a $C^*$-diagonal. This follows from a fibrewise argument analogous to the proof of \cite[Proposition~4.1]{BR}.
\end{remark}

\begin{proposition}\label{weiyi}
For every $p$, the algebra $\tilde{\mathbb{I}}_p$ has a unique Cartan subalgebra up to conjugacy.
\end{proposition}

\begin{proof}
Here $K_{p,1}=\operatorname{Sym}(p)$, so \cref{fenlei} gives one non-degenerate class.  Moreover, $\gcd(1,p)=1$, so \cref{feituihua} shows that every Cartan
subalgebra is non-degenerate.
\end{proof}

\begin{proposition}\label{wuge}
The algebra $M_2(\widetilde{\mathbb I}_2)$ has exactly five conjugacy classes of non-degenerate Cartan subalgebras.
\end{proposition}

\begin{proof}
We identify $\underline 2\times\underline 2$~with $\{1,2,3,4\}$ by setting
\[1=(1,1),\qquad2=(2,1),\qquad3=(1,2),\qquad4=(2,2).\]
Under this identification,~$\operatorname{Sym}(\underline{2}\times \underline{2})\cong S_4$ and $K_{2,2}=\{e,(12)(34),(13)(24),(14)(23)\}$. Thus $K_{2,2}$ is a normal subgroup of $S_4$. We next identify the quotient $S_4/K_{2,2}$.  Consider the three partitions of $\{1,2,3,4\}$ into two unordered pairs:
\[\mathcal P_1=\{\{1,2\},\{3,4\}\},\qquad\mathcal P_2=\{\{1,3\},\{2,4\}\},\qquad\mathcal P_3=\{\{1,4\},\{2,3\}\}.\]
We therefore obtain a group homomorphism $\theta:S_4\longrightarrow \operatorname{Sym}\{\mathcal P_1,\mathcal P_2,\mathcal P_3\} \cong S_3.$ Obviously, $\ker\theta=K_{2,2}$. By the first isomorphism theorem, we have $S_4/K_{2,2}\cong S_3$. The six elements of $S_4/K_{2,2}\cong S_3$ may be represented by
\[e,\qquad (12),\qquad (13),\qquad (14),\qquad(123),\qquad (132). \]
	
By \cref{fenlei}, two twists determine conjugate Cartan subalgebras precisely when their cosets in $S_4/K_{2,2}$ are equal or inverse.  We must therefore compute the orbits of the inversion map $xK_{2,2}\longmapsto x^{-1}K_{2,2}$. The identity coset is fixed. Each of the three transposition cosets is also fixed. On the other hand, $(123)^{-1}=(132)$, so the two three-cycle cosets form one inversion orbit. Therefore $M_2(\widetilde{\mathbb I}_2)$ has exactly five conjugacy classes of non-degenerate Cartan subalgebras.
\end{proof}

\begin{remark}
Proposition~\ref{wuge} gives a counterexample to \cite[Theorem~5.10]{BR}. Indeed, let $F\in U(4)$ be the tensor flip and choose a continuous path $w\colon[0,1]\to U(4)$ with $w(0)=F$ and $w(1)=1$. Then $\operatorname{Ad}w$ identifies $M_2(\widetilde{\mathbb I}_2)$ with the algebra $I_{2,2}$ of \cite{BR}. The equivalence relation stated there yields two non-degenerate Cartan conjugacy classes, whereas Proposition~\ref{wuge} gives five.
	
The issue is that a permutation preserving the endpoint diagonal need not normalize the full endpoint algebra. For example, put
\[ s=\begin{pmatrix}0&1\\1&0\end{pmatrix},
\qquad P=s\otimes e_{11}+1_2\otimes e_{22}. \]
Then
\[ P(1_2\otimes D_2)P^*=1_2\otimes D_2, \]
but
\[ P(1_2\otimes e_{12})P^*=s\otimes e_{12}\notin1_2\otimes M_2. \]
Thus the endpoint permutation groups in the cited theorem must be restricted to preserve the full endpoint algebras.
\end{remark}

\subsection{Spectra of non-degenerate Cartan subalgebras}

\begin{theorem}\label{pu}
Let $\sigma\in\operatorname{Sym}(\underline{p}\times \underline{m})$.  Then $\operatorname{Spec}(B_\sigma)$ is homeomorphic to
\[
 X_\sigma=
 \bigl(\underline{p}\times \underline{m}\times([0,1]_-\sqcup[1,2]_+)\bigr)/{\sim},
\]
where $\sim$ is generated by
\begin{align*}
 (i,j,0)&\sim(i',j,0),\qquad i,i'\in\underline p,\ j\in\underline m,\\
 (i,j,2)&\sim(i',j,2),\qquad i,i'\in\underline p,\ j\in\underline m,\\
 (\sigma(i,j),1_-)&\sim(i,j,1_+),\qquad (i,j)\in\underline{p}\times \underline{m}.
\end{align*}
\end{theorem}

\begin{proof}
Over either outer interval, the spectrum consists of $pm$ open strands whose
$p$ points with the same second coordinate are identified at the endpoint. On the middle interval, the minimal projection field indexed by $(i,j)$ is
\[ u_t^\sigma(e_{ii}\otimes e_{jj})(u_t^\sigma)^*. \]
At $1/3$ its label is $\sigma(i,j)$, and at $2/3$ its label is $(i,j)$.  These are exactly the displayed identifications.
\end{proof}

\begin{definition}
In the model \cref{pu}, write $(\ast,j,0)$ and $(\ast,j,2)$ for the endpoint equivalence classes.
\end{definition}

\begin{definition}\label{yuehua}
Let $\sigma\in\operatorname{Sym}(\underline{p}\times \underline{m})$.  Its permutation matrix is indexed by $\underline p\times\underline m$ and has entries
\[ \sigma_{(i,j),(i',j')}=\delta_{\sigma(i',j'),(i,j)}. \]
The reduced matrix is
$R(\sigma)=(R_{jk}(\sigma))\in M_m(\mathbb{Z}_{\geq0})$, where
\begin{equation}\label{yuehuashi}
R_{jj'}(\sigma) =\sum_{i,i'\in\underline p}\sigma_{(i,j),(i',j')}.
\end{equation}
\end{definition}
For convenience, let~$a,b,c,d\in\mathbb{N}$, set
\[
\mathcal M(a,b,c,d)
=\left\{A\in M_{a,c}(\mathbb Z_{\geq0}):
\begin{array}{l}
	\text{every row sum of $A$ equals $b$, and}\\
	\text{every column sum of $A$ equals $d$}
\end{array}\right\}.
\]
Let $\Gamma_\sigma$ be the bi-partite multigraph with row vertices $v_1,\ldots,v_m$, column vertices $w_1,\ldots,w_m$, and $R_{jk}(\sigma)$ edges from $v_j$ to $w_k$. Let $A,B\in M_{m,n}(\mathbb{Z}_{\geq0})$ be two matrices and consider their associated bi-partite graphs $\Gamma$, $\Lambda$, respectively. Permuting rows or columns merely relabels the vertices on the corresponding side of the bi-partite multigraph, so its isomorphism class does not change.

\begin{definition}
The geometric realization $|\Gamma|$ of a graph $\Gamma$ is the one-dimensional CW complex obtained by taking the vertices as $0$-cells and attaching one $1$-cell for every edge according to its incidence relation.
\end{definition}

\begin{theorem}\label{putu}
Let $\sigma\in\operatorname{Sym}(\underline{p}\times \underline{m})$ and let $\Gamma_\sigma$ be the bi-partite graph whose adjacency matrix is the reduced matrix of $\sigma$. Then
\[
\operatorname{Spec}(B_\sigma)\cong|\Gamma_\sigma|.
\]
\end{theorem}

\begin{proof}
Let $R(\sigma)$ be the reduced matrix of $\sigma$ and let $X=\operatorname{Spec}(B_\sigma)$. By \cref{pu}, the $0$-cells of $X$ are
\[ \{(\ast,j,0):j\in\underline m\}\cup \{(\ast,j,2):j\in\underline m\}. \]
Write $\sigma(i,j)=(\sigma_1(i,j),\sigma_2(i,j))$. The $1$-cell indexed by $(i,j)\in\underline{p}\times \underline{m}$ is
\[
\{\sigma(i,j)\}\times[0,1]_-
\ \cup\ 
\{(i,j)\}\times[1,2]_+,
\]
and joins $(\ast,\sigma_2(i,j),0)$ to $(\ast,j,2)$. The bi-partite graph $\Gamma_\sigma$ has vertices
\[
\{(j,r):j\in\underline m\}\cup\{(j,c):j\in\underline m\},
\]
with $R_{jj'}(\sigma)$ edges from $(j,r)$ to $(j',c)$. Map $(\ast,j,0)$ to $(j,r)$ and $(\ast,j,2)$ to $(j,c)$. The number of $1$-cells of $X$ joining $(\ast,j,0)$ to $(\ast,j',2)$ is
\begin{align*}
\bigl|\{(i',j'):\sigma_2(i',j')=j\}\bigr|
&=\sum_{i'\in\underline p}\delta_{\sigma_2(i',j'),j}\\
&=\sum_{i,i'\in\underline p}
\delta_{\sigma(i',j'),(i,j)}
=R_{jj'}(\sigma).
\end{align*}
Thus the vertex bijection preserves the number of edges between every pair of vertices and extends to an isomorphism of CW complexes.
\end{proof}

The reduced matrix of every $\sigma\in\operatorname{Sym}(\underline p\times\underline m)$ belongs to $\mathcal M(m,p,m,p)$. Conversely, every matrix in this set occurs.

\begin{theorem}\label{keda}
For every $R\in\mathcal M(m,p,m,p)$, there exists $\sigma\in\operatorname{Sym}(\underline p\times\underline m)$ whose reduced matrix is $R$.
\end{theorem}

\begin{proof}
Write $R=(R(a,b))_{a,b\in\underline m}$. For each $b\in\underline m$, put $S_b=\{(i,b):i\in\underline p\}$. Since the $b$-th column sum of $R$ is $p$, choose a partition 
\[ S_b=\bigsqcup_{a=1}^{m}S_{a,b}, \qquad
|S_{a,b}|=R(a,b).\]
Similarly, for each $a\in\underline m$, put $T_a=\{(i,a):i\in\underline p\}$. Since the $a$-th row sum of $R$ is $p$, choose a partition
\[ T_a=\bigsqcup_{b=1}^{m}T_{a,b},\qquad
|T_{a,b}|=R(a,b). \] 
For every $a,b$, choose a bijection $\theta_{a,b}\colon S_{a,b}\longrightarrow T_{a,b}$. The disjoint union of these bijections defines a permutation
\[ \sigma=\bigsqcup_{a,b}\theta_{a,b}
\in\operatorname{Sym}(\underline p\times\underline m). \]
For $a,b\in\underline m$, we then have
\[ \#\{i\in\underline p: \operatorname{pr}_2(\sigma(i,b))=a\} =|S_{a,b}| =R(a,b).
\]
Hence the reduced matrix of $\sigma$ is $R$.
\end{proof}

\begin{proposition}\label{m2}
The spectra of non-degenerate Cartan subalgebras of $M_m(\tilde{\mathbb{I}}_2)$ have exactly $m$ different homeomorphism types.
\end{proposition}

\begin{proof}
By \cref{biaozhunxing}, $B$ is conjugate to $B_\sigma$ for some $\sigma\in\operatorname{Sym}(\underline{2}\times \underline{m})$. Conjugacy preserves the spectrum, and hence $\operatorname{Spec}(B)\cong\operatorname{Spec}(B_\sigma)$. By \cref{putu}, $\operatorname{Spec}(B_\sigma)\cong|\Gamma_\sigma|$.

Recall that the row and column sums of the reduced matrix $R(\sigma)$ are equal to $p=2$. It suffices to count the matrix in $\mathcal M(m,2,m,2)$. If $A$ lies in this set, then each entry is $0$, $1$, or $2$.  By permuting rows and columns, every entry equal to $2$ may be placed on the diagonal, and the other entries in its row and column must vanish. Starting from an entry equal to $1$ and continuing alternately along rows and columns produces one of the blocks
\[
(2),\qquad
\begin{pmatrix}1&1\\1&1\end{pmatrix},\qquad
\begin{pmatrix}1&1&0\\1&0&1\\0&1&1\end{pmatrix},\qquad
\begin{pmatrix}
	1&1&0&0\\1&0&1&0\\0&1&0&1\\0&0&1&1
\end{pmatrix},\quad\ldots
\]
Each block has geometric realization a circle. Row and column permutations put $A$ into block diagonal form. Consequently, the spectrum of a non-degenerate Cartan subalgebra of $M_m(\tilde{\mathbb{I}}_2)$ is a disjoint union of $k$ circles for some
$k\in\{1,\ldots,m\}$.
\end{proof}

\begin{remark}\label{liangzhong}
For $m=p=2$, the five conjugacy classes of non-degenerate Cartan subalgebras found in Proposition~\ref{wuge} give precisely two homeomorphism types of spectra. More precisely,
\[
\operatorname{Spec}(B_e)
\cong
\operatorname{Spec}(B_{(12)})
\cong S^1\sqcup S^1,
\]
whereas
\[
\operatorname{Spec}(B_{(13)})
\cong
\operatorname{Spec}(B_{(14)})
\cong
\operatorname{Spec}(B_{(123)})
\cong S^1.
\]
Thus the homeomorphism type of the spectrum is not a complete invariant for conjugacy in $M_2(\tilde{\mathbb{I}}_2)$.
\end{remark}

\section{Prescribed finite numbers of Cartan conjugacy classes}
In view of the issue with the endpoint permutation groups discussed above, we revisit the finite enumeration question considered in \cite{BR}. We give an independent construction showing that, for every $n\in\mathbb{N}$, there exists a
subhomogeneous $\mathrm{C}^*$-algebra with exactly $n$ conjugacy classes of Cartan subalgebras.

\begin{definition}
Fix integers $1\leq a<b$, set $N=a+b$, and regard $F_{a,b}=M_a\oplus M_b$ as the usual block-diagonal subalgebra of $M_N$. The splitting interval algebra considered here is
\[
S_{a,b}
=
\left\{
f\in C([0,1],M_N): f(0),f(1)\in F_{a,b}
\right\}.
\]
\end{definition}
Like dimension drop algebras and the dimension drop interval algebras introduced above, it is also a one-dimensional NCCW complex. A Cartan subalgebra $B\subseteq S_{a,b}$ is called non-degenerate if $\dim B(0)=\dim B(1)=N$. The following lemma proves that non-degeneracy is automatic when the two block sizes are different.

\begin{lemma}\label{duandian}
Every Cartan subalgebra of $S_{a,b}$ is non-degenerate.
\end{lemma}

\begin{proof}
Let $B\subseteq S_{a,b}$ be a Cartan subalgebra. We prove that $B(0)$ is a MASA of $M_a\oplus M_b$. Put $A=S_{a,b}$, $C=B(0)$. Fix $t\in(0,1)$ and choose a closed interval $I\subseteq(0,1)$ whose interior contains $t$. Then $B|_I$ is a Cartan subalgebra of $C(I,M_N)$. By \cref{qici}, $B|_I$ is conjugate to $C(I,D_N)$. In particular, $B(t)$ is a MASA of $M_N$.

Choose a sequence $t_j\rightarrow0$. For each $j$, choose $u_j\in U(N)$ such that $B(t_j)=u_jD_Nu_j^*$. After passing to a subsequence, compactness of $U(N)$ allows us to assume that $u_j\to u\in U(N)$. Set $D=uD_Nu^*$, then $D$ is a MASA of $M_N$. For $x\in F_{a,b}$, let $\widehat{x}\in A$ denote the constant section with value $x$. Continuity of $E(\widehat{x})$ gives
\[
\begin{aligned}
	E_0(x)
	=E(\widehat{x})(0)
	=\lim_{j\to\infty}E(\widehat{x})(t_j) 
	=\lim_{j\to\infty}E_{t_j}(x)
	=E_D(x),
\end{aligned}
\]  
where $E_D\colon M_N\to D$ denotes the canonical expectation and $E_0\colon F_{a,b}\to C$ is the conditional expectation induced at the endpoint. Hence $E_0=E_D|_{F_{a,b}}$. It follows that $E_0$ is trace-preserving and that $C\subseteq D$.

Let $p=(p_a,p_b)$ be a minimal projection of $C$, and set $r=\operatorname{rank}(p_a)$, $s=\operatorname{rank}(p_b)$. For $x\in pF_{a,b}p$, bimodularity of $E_0$ and minimality of $p$ give $E_0(x)\in pCp=\mathbb Cp$. Since $E_0$ preserves the ordinary matrix trace, we obtain $	E_0(x)=\frac{\operatorname{Tr}(x)}{r+s}\,p$. Since $p\in C\subseteq D$, choose a rank-one projection $q\in D$ with $q\leq p$, and write $q=\xi\xi^*$ for a unit vector $\xi$. Then 
\[ \begin{aligned}
		(\xi^*x\xi)q
		=qxq
		=E_D(qxq)
		=qE_D(x)q  
		=qE_0(x)q
		=\frac{\operatorname{Tr}(x)}{r+s}\,q.
\end{aligned}\]
As $q\neq0$, it follows that \[\xi^*x\xi=\frac{\operatorname{Tr}(x)}{r+s}.\tag{1}\]

Write $\xi=(\xi_a,\xi_b)$ according to $\mathbb C^N=\mathbb C^a\oplus\mathbb C^b$. If $r\geq2$, choose a unit vector $\eta\in p_a\mathbb C^a$ orthogonal to $\xi_a$, and put $x=(\eta\eta^*,0)\in pF_{a,b}p$. Then $\xi^*x\xi=0$, whereas
\[\frac{\operatorname{Tr}(x)}{r+s}=\frac{1}{r+s}>0,\]
contradicting (1). Thus $r\leq1$. The same argument in the second summand gives $s\leq1$. Consequently, every minimal projection of $C$ has one of the rank profiles $(1,0)$, $(0,1)$, $(1,1)$.

We next exclude the profile $(1,1)$. Let $p,p'$ be minimal projections of $C$, let $n\in N_A(B)$, and set $T=pn(0)p'$. Since $n(0)\in N_{F_{a,b}}(C)$, the element $T$ also normalizes $C$. If $T\neq0$, then minimality of $p$ and $p'$ gives $T^*T=\lambda p'$, $TT^*=\mu p$ for some $\lambda,\mu>0$. Write
\[T=(T_a,T_b),\qquad p=(p_a,p_b),\qquad p'=(p'_a,p'_b). \] Then $T_a^*T_a=\lambda p'_a$, $T_aT_a^*=\mu p_a$. Since $T_a^*T_a$ and $T_aT_a^*$ have the same rank, we obtain $\operatorname{rank}(p'_a)=\operatorname{rank}(p_a)$. The same argument in the second summand gives $\operatorname{rank}(p'_b)=\operatorname{rank}(p_b)$. Thus $T\neq0$ implies that $p$ and $p'$ have the same rank profile.
	
Suppose that $p$ has profile $(1,1)$, while $p'$ has profile $(1,0)$ or $(0,1)$. The preceding paragraph then gives $pn(0)p'=0$ for every $n\in N_A(B)$. On the other hand, $pF_{a,b}p'\neq0$, since $p$ and $p'$ have a common nonzero matrix summand. Regularity of $B$ gives $F_{a,b} =\overline{\operatorname{span}} \{n(0):n\in N_A(B)\}$. Multiplying this equality on the left by $p$ and on the right by $p'$ yields $pF_{a,b}p'=0$, a contradiction. It follows that if one minimal projection of $C$ had profile $(1,1)$, then every minimal projection of $C$ would have that profile. If there were $m$ such minimal projections, summing them to $1_{F_{a,b}}$ would give $a=m=b$, contrary to $a<b$. Hence the profile $(1,1)$ cannot occur.
	
Every minimal projection of $C$ therefore has profile $(1,0)$ or $(0,1)$, and hence has total rank one. Since the minimal projections of $C$ sum to $1_{F_{a,b}}$, there are exactly $a+b=N$ of them. Thus $C=B(0)$ is a MASA of $F_{a,b}$. Repeating the same argument at $1$ shows that $B(1)$ is also a MASA of $F_{a,b}$. Therefore $B$ is non-degenerate.
\end{proof}

\begin{lemma}\label{lianxu}
Let $B\subseteq S_{a,b}$ be Cartan. There exist continuous rank-one projection sections $p_1,\ldots,p_N\in B$, pairwise orthogonal and summing to $1$, such that
\begin{equation*}
B= \left\{
\sum_{j=1}^{N}h_jp_j:h_j\in C([0,1])
\right\}.
\label{lianxushi}
\end{equation*}
In particular, $B\cong C([0,1])^{\oplus N}$,	and $B$ is a $C^*$-diagonal. Consequently, all Cartan subalgebras of $S_{a,b}$ are mutually isomorphic.
\end{lemma}

\begin{proof}
By  Lemma~\ref{duandian} and \cref{qici}, $B(t)$ is an $N$-dimensional MASA of $M_N$ for every $t\in[0,1]$. Fix $t_0\in[0,1]$, and let $e_1,\ldots,e_N$ be the minimal projections of $B(t_0)$. Choose pairwise distinct real numbers $\lambda_1,\ldots,\lambda_N$ and set
\[c_0=\sum_{j=1}^N \lambda_j e_j\in B(t_0)_{\mathrm{sa}}.\]
There exists $c\in B_{\mathrm{sa}}$ such that $c(t_0)=c_0$. By continuity of $c$ and continuous functional calculus, there is a connected neighbourhood $U$ of $t_0$ and continuous projection sections $q_1,\ldots,q_N\in B|_U$ such that
\[q_j(t_0)=e_j,\qquad q_jq_k=0\ (j\neq k),\qquad \sum_{j=1}^N q_j=1. \]
Indeed, one chooses pairwise disjoint intervals around the numbers $\lambda_j$ and takes the corresponding spectral projections of $c(t)$. The rank of a norm-continuous projection is locally constant; after shrinking $U$ if necessary, each $q_j(t)$ therefore has rank one. Hence $q_1(t),\ldots,q_N(t)$ are precisely the minimal projections of $B(t)$ and span $B(t)$ for every $t\in U$.
	
If two such local families are defined on overlapping neighbourhoods, then at each point of the overlap they are two enumerations of the minimal projections of the same MASA. They therefore differ by a permutation, and this permutation is locally constant. Consequently, the minimal projections of the fibres $B(t)$ form an $N$-sheeted covering of $[0,1]$. Since $[0,1]$ is contractible, this covering is trivial. We may thus label its sheets by continuous rank-one projection sections $p_1,\ldots,p_N$, which are pairwise orthogonal and sum to $1$.
	
Choose a finite open cover $(U_\alpha)_\alpha$ as above. After relabelling the local spectral projections, for every $\alpha$ and $j$ there is an element $b_{\alpha,j}\in B$ such that $b_{\alpha,j}(t)=p_j(t)$, $t\in U_\alpha$. We may choose continuous functions $\rho_\alpha\colon [0,1]\to[0,1]$ such that
\[
\operatorname{supp}(\rho_\alpha)\subseteq U_\alpha
\quad\text{for every }\alpha,
\qquad
\sum_\alpha \rho_\alpha(t)=1
\quad\text{for every }t\in[0,1].
\] 
The sum $\sum_\alpha \rho_\alpha b_{\alpha,j}$ belongs to $B$.
Pointwise, it equals $p_j$, and hence $p_j\in B$ for every $j$.
	
Let $b\in B$. Since $p_1(t),\ldots,p_N(t)$ span $B(t)$, there are unique scalars $h_j(t)$ such that $b(t)=\sum_{j=1}^N h_j(t)p_j(t)$. Because $p_j(t)$ has rank one, these coefficients are given by $h_j(t)=\operatorname{Tr}\bigl(b(t)p_j(t)\bigr)$, where $\operatorname{Tr}$ denotes the ordinary matrix trace. Thus each $h_j$ is continuous. Conversely, if $h_j\in C([0,1])$, then $h_jp_j\in B$ because $C([0,1])1_N\subseteq B$. It follows that
\[ B = \left\{ \sum_{j=1}^N h_jp_j : h_j\in C([0,1]) \right\}, \]
and the map
\[ C([0,1])^{\oplus N}\longrightarrow B, \qquad (h_1,\ldots,h_N)\longmapsto \sum_{j=1}^N h_jp_j, \]
is a $*$-isomorphism.

Finally, every pure state of $B$ is of the form $\psi\circ\operatorname{ev}_t$, where $t\in[0,1]$ and $\psi$ is a pure state of $B(t)$. Any state extension restricts to evaluation at $t$ on $C([0,1])$, and hence factors through the fibre $(S_{a,b})_t$. Since $B(t)\subseteq (S_{a,b})_t$ is a finite-dimensional diagonal for every $t\in[0,1]$, it has the unique extension property. Thus $B$ has the unique extension property and, being Cartan, is a $C^*$-diagonal.
\end{proof}

The preceding lemma reduces the remaining analysis to the way in which the $N$ continuous diagonal strands meet the two matrix summands at the endpoints. We now introduce the corresponding endpoint invariant and use it to compute the Cartan conjugacy classes explicitly. Put $z_1=1_a\oplus0_b$, $z_2=0_a\oplus1_b$ in $F_{a,b}$. For $B$ and its strands as in Lemma~\ref{lianxu}, define the endpoint matching matrix $R(B)=(r_{ij})_{i,j=1}^{2}$ by
\begin{equation*}
r_{ij} =\#\{\ell:p_\ell(0)\leq z_i,\ p_\ell(1)\leq z_j\}. \label{pipei}
\end{equation*}
Each endpoint projection lies in exactly one of the two matrix summands. The row and column sums of $R(B)$ are both $(a,b)$. Thus
\begin{equation*}
R(B)=R_k:=
	\begin{pmatrix}
		a-k & k\\
		k & b-k
	\end{pmatrix}
	\qquad
	\text{for a unique }k\in\{0,\ldots,a\}.
	\label{pipeishi}
\end{equation*}
The matrix records the full endpoint block of each strand. Its entries do
not identify distinct diagonal characters: the spectrum remains a disjoint union of $N$ intervals.

\begin{proposition} \label{wanbei}
For Cartan subalgebras $B,C\subseteq S_{a,b}$, the following conditions are equivalent:
\begin{enumerate}
\item an automorphism of $S_{a,b}$ carries $B$ onto $C$;
\item a unitary in $S_{a,b}$ conjugates $B$ onto $C$;
\item $R(B)=R(C)$.
\end{enumerate}
Every matrix $R_k$, $0\leq k\leq a$, occurs. Hence $S_{a,b}$ has exactly $a+1$ Cartan conjugacy classes, all non-degenerate.
\end{proposition}

\begin{proof}
An automorphism of $S_{a,b}$ acts on its centre by a homeomorphism of $[0,1]$. It therefore either preserves the two endpoints individually or interchanges them. The induced endpoint isomorphisms preserve the two matrix summands because $a\neq b$. It follows that the matching matrix is unchanged in the first case and transposed in the second. Since $R_k$ is symmetric, (1) implies (3).
	
Assume (3). Label the strands $p_1,\ldots,p_N$ of $B$ and $q_1,\ldots,q_N$ of $C$ so that corresponding strands have the same pair of endpoint blocks. Choose  continuous maps $\xi_j,\eta_j\colon[0,1]\longrightarrow\mathbb{C}^N$ such that, for every $t\in[0,1]$, $\xi_j(t)$ and $\eta_j(t)$ are unit vectors spanning $\operatorname{ran}(p_j(t))$ and $\operatorname{ran}(q_j(t))$, respectively. The formula $v(t)=\sum_{j=1}^{N}\eta_j(t)\xi_j(t)^*$ defines a continuous unitary-valued map satisfying $v(t)p_j(t)v(t)^*=q_j(t)$ for every $j$ and every $t\in[0,1]$. At either endpoint it maps each of the two coordinate block subspaces onto itself. Hence $v(0),v(1)\in F_{a,b}$, so $v\in U(S_{a,b})$ and $\operatorname{Ad}(v)(B)=C$. Thus (3) implies (2), while (2) trivially implies (1).

For realization, label the first block by $1,\ldots,a$ and the second by $a+1,\ldots,a+b$. For $0\leq k\leq a$, put 
\[\sigma_k=(1,a+1)(2,a+2)\cdots(k,a+k)\in\operatorname{Sym}(N), \qquad \sigma_0=\operatorname{id}. \]
Choose a continuous path $u_k:[0,1]\to U(N)$ which is equal to 	$P_{\sigma_k}$ near $0$ and to $1$ near $1$, and set $B_k=u_kC([0,1],D_N)u_k^*$. It is straightforward to verify that this is a Cartan subalgebra of $S_{a,b}$, and its strands have matching matrix $R_k$.
\end{proof}

\begin{theorem} \label{zhiding}
For every $n\in\mathbb{N}$, there exists a subhomogeneous
$C^*$-algebra with exactly $n$ conjugacy classes of Cartan
subalgebras. Moreover, every Cartan subalgebra of this algebra is a
$C^*$-diagonal.
\end{theorem}

\begin{proof}
For $n=1$, the dimension drop algebra $\widetilde{\mathbb I}_p$ has a unique Cartan subalgebra up to conjugacy; see Proposition~\ref{weiyi}. For $n>1$, the splitting interval algebra $S_{n-1,n}$ has exactly $n$ Cartan subalgebras up to conjugacy by Proposition~\ref{wanbei}.
\end{proof}
\begin{remark}
The endpoint permutation group in this construction is
\[ K=\operatorname{Sym}(a)\times\operatorname{Sym}(b)
\leq\operatorname{Sym}(a+b),\]
acting separately on the two coordinate blocks. It consists exactly of the permutations normalizing the full algebra $M_a\oplus M_b$, since the two summands have different sizes. The matrices $R_k$ classify $K\backslash\operatorname{Sym}(a+b)/K$, and inversion fixes each double coset. Thus the finite count uses the full endpoint algebra.
\end{remark}

\section{Existence and nonuniqueness of AD Cartan subalgebras} \label{guina}

\subsection{Existence of AD Cartan subalgebras}\

We write $A=\varinjlim(A_n,\varphi_n)$ for an inductive limit of $\mathrm{C}^*$-algebras with connecting maps $\varphi_n\colon A_n\to A_{n+1}$.

\begin{definition}	
An AD algebra is an inductive limit $A=\varinjlim(A_n,\varphi_n)$ in which each $A_n$ is a finite direct sum of dimension-drop interval algebras and $M_k(C(X))$, where $X$ is one of the finite connected CW complexes $\{\mathrm{pt}\}$, $\mathbb{T}$, and $[0,1]$.  If $B_n\subseteq A_n$ are Cartan subalgebras, $\varphi_n(B_n)\subseteq B_{n+1}$, and $B=\varinjlim(B_n,\varphi_n|_{B_n})$ is Cartan in $A$, then $B$ is called
an AD Cartan subalgebra of $A$.
\end{definition}

We use Li's criterion for passing Cartan subalgebras to inductive limits.

\begin{lemma}[{\cite[Theorem~1.10]{L2}}]\label{lijixian}
Let $(A_n,B_n)$ be Cartan pairs, $N_n=N_{A_n}(B_n)$, and let $P_n:A_n\to B_n$ be faithful conditional expectations.  Suppose $\varphi_n:A_n\to A_{n+1}$ are injective $*$-homomorphisms satisfying
\[ \varphi_n(B_n)\subseteq B_{n+1},\qquad
 \varphi_n(N_n)\subseteq N_{n+1},\qquad P_{n+1}\varphi_n=\varphi_nP_n. \]
Then $\varinjlim B_n$ is a Cartan subalgebra of $\varinjlim A_n$. If all $B_n$ are $\mathrm{C}^*$-diagonals, then $\varinjlim B_n$ is a $\mathrm{C}^*$-diagonal of $\varinjlim A_n$.
\end{lemma}

We now describe the inductive system used below. Put $d_n=3^n$ and $A_n=M_{d_n}(\tilde{\mathbb{I}}_2)$.  Let $(x_n)$ be a dense sequence in $(0,1)$. There is a permutation unitary $V_n$ which rearranges the three summands so that, for $x\in M_2\otimes M_{d_n}$ and $a\in M_{d_n}$,
\begin{equation*}
 V_n\operatorname{diag}(x,1_2\otimes a,x)V_n^*
 =1_2\otimes\operatorname{diag}(x,a).
\end{equation*}
Define
\begin{equation*}
 \varphi_n(f)(t)
 =V_n\operatorname{diag}(f(x_n),f(t),f(x_n))V_n^*.
\end{equation*}
Clearly, $\varphi_n$ is well-defined and injective. Fix $\sigma\in\operatorname{Sym}(\underline2\times\underline3)$ and an implementing unitary $U_1=u^\sigma$. Recursively set
\begin{equation*}
 U_{n+1}(t)=V_n\operatorname{diag}(U_n(x_n),U_n(t),U_n(x_n))V_n^*
\end{equation*}
and
\begin{equation*}
 B_n^\sigma
 =U_nC([0,1],D_{2d_n})U_n^*\cap A_n.
\end{equation*}

\begin{lemma}\label{fenjie}
For $\epsilon\in\{0,1\}$, there are $W_{n,\epsilon}\in U(d_n)$ and a permutation unitary $S_{n,\epsilon}\in M_{2d_n}$ such that $	U_n(\epsilon) = (1_2\otimes W_{n,\epsilon})S_{n,\epsilon}$.
Consequently, $B_n^\sigma(\epsilon) = 1_2\otimes W_{n,\epsilon}D_{d_n}W_{n,\epsilon}^*.$
\end{lemma}

\begin{proof}
At $n=1$, take
\[W_{1,0}=W_{1,1}=1,\qquad S_{1,0}=P_\sigma,\qquad S_{1,1}=1. \]
Suppose it holds at stage $n$, then
\[ V_n\operatorname{diag} \bigl(U_n(x_n),1_2\otimes W_{n,\epsilon},U_n(x_n)\bigr)V_n^* = 1_2\otimes \operatorname{diag}\bigl(U_n(x_n),W_{n,\epsilon}\bigr). \]
The remaining factor $	S_{n+1,\epsilon} = V_n\operatorname{diag}(1,S_{n,\epsilon},1)V_n^*$ is a permutation unitary. This proves the induction and the formula for the endpoint diagonal. Indeed, writing $U_n(\epsilon)=(1_2\otimes W)S$ gives $U_n(\epsilon)D_{2d_n}U_n(\epsilon)^* \cap(1_2\otimes M_{d_n}) = 1_2\otimes WD_{d_n}W^*.$
\end{proof}

If the permutation underlying $S_{n,\epsilon}$ is denoted by the same letter, set $q_{n,\epsilon}(j) =\operatorname{pr}_2(S_{n,\epsilon}j)$, $j\in \underline{2}\times \underline{d_n}$, $\epsilon=0,1$, and define
\begin{equation*}
R_n(a,b) = \#\bigl\{ j\in\underline{2}\times \underline{d_n}: q_{n,0}(j)=a,\ q_{n,1}(j)=b \bigr\}, \qquad a,b\in\underline{d_n}.
\end{equation*}

\begin{lemma}\label{zuobiao}
The spectrum of $B_n^\sigma$ is the realization of the bi-partite graph with edge set $\underline{2}\times \underline{d_n}$, in which $j$ has left endpoint $q_{n,0}(j)$ and right endpoint $q_{n,1}(j)$. Its reduced matrix is $R_n$, and $R_1=R_\sigma$.
\end{lemma}

\begin{proof}
At the endpoint $\epsilon$, Lemma~\ref{fenjie} gives
\[ U_n(\epsilon)^* \bigl( 1_2\otimes W_{n,\epsilon}e_{aa}W_{n,\epsilon}^* \bigr) U_n(\epsilon) = S_{n,\epsilon}^* (1_2\otimes e_{aa}) S_{n,\epsilon}. \]
Thus, the $\epsilon$-endpoint of the edge indexed by $j$ is the vertex $a$ if and only if $q_{n,\epsilon}(j)=a$. This proves the graph description. At $n=1$, $S_{1,0}=P_\sigma$ and $S_{1,1}=1$, so $R_1=R_\sigma$.
\end{proof}

\begin{lemma}\label{xiangrong}
For every $n$, the subalgebra $B_n^\sigma$ is a non-degenerate $C^*$-diagonal of $A_n$. Moreover, $\varphi_n(B_n^\sigma)\subseteq B_{n+1}^\sigma$ and $\varphi_n\bigl(N_{A_n}(B_n^\sigma)\bigr) \subseteq N_{A_{n+1}}(B_{n+1}^\sigma)$.
The maps
\begin{equation*}
E_n = \operatorname{Ad}(U_n)\circ E_{D_{2d_n}} \circ\operatorname{Ad}(U_n^*)
\end{equation*}
restrict to faithful conditional expectations $E_n\colon A_n\longrightarrow B_n^\sigma$ and satisfy $E_{n+1}\circ\varphi_n=\varphi_n\circ E_n$.
\end{lemma}

\begin{proof}
For every $t\in(0,1)$, the fibre $B_n^\sigma(t)$ is the conjugate by $U_n(t)$ of the standard diagonal $D_{2d_n}$. At either endpoint, Lemma~\ref{fenjie} gives $B_n^\sigma(\epsilon) = 1_2\otimes W_{n,\epsilon}D_{d_n}W_{n,\epsilon}^*$, $\epsilon\in\{0,1\}$, which is a maximal abelian subalgebra of $1_2\otimes M_{d_n}$. In particular, $B_n^\sigma$ is non-degenerate.
	
We first verify that the expectation preserves the endpoint conditions. Suppose that $f(\epsilon)=1_2\otimes a$ and $U_n(\epsilon)=(1_2\otimes W)S$. Since $S$ is a permutation unitary, direct computation gives
\[ E_n(f)(\epsilon) = 1_2\otimes W E_{D_{d_n}}(W^*aW)W^*. \]
Thus $E_n(f)\in A_n$, and hence $E_n$ restricts to a faithful conditional expectation from $A_n$ onto $B_n^\sigma$.
	
We next check regularity. Let $e_{jk}\in M_{d_n}$ be a matrix unit. The element $S^*(1_2\otimes e_{jk})S$ is partial monomial with respect to $D_{2d_n}$. If $\chi$ is supported in a sufficiently small neighbourhood of the endpoint $\epsilon$, then
\[ t\longmapsto \chi(t)U_n(t)S^*(1_2\otimes e_{jk})S U_n(t)^* \]
defines an element of $A_n$ normalizing $B_n^\sigma$. These endpoint normalizers, together with the compactly supported interior matrix units and central cutoff functions, generate $A_n$. Therefore, $B_n^\sigma$ is regular in $A_n$. Its fibrewise maximal abelianness also shows that it is a MASA, so $B_n^\sigma$ is a Cartan subalgebra of $A_n$. It follows from Remark~\ref{duijiao} that it is a $C^*$-diagonal.
	
A direct verification from the definitions, applied to the three assertions in turn, yields
\[
\varphi_n(B_n^\sigma)\subseteq B_{n+1}^\sigma,
\qquad
\varphi_n\bigl( N_{A_n}(B_n^\sigma)\bigr)
\subseteq
 N_{A_{n+1}}(B_{n+1}^\sigma),
\qquad
E_{n+1}\circ\varphi_n=\varphi_n\circ E_n. \qedhere
\]
\end{proof}

\begin{proposition}\label{cunzai}
For every $\sigma$, the algebra
$B^\sigma=\varinjlim(B_n^\sigma,\varphi_n|_{B_n^\sigma})$ is an AD Cartan
subalgebra of $\mathcal{A}=\varinjlim(A_n,\varphi_n)$. $B^\sigma$ is a $C^*$-diagonal.
\end{proposition}

\begin{proof}
The result follows immediately from Lemma~\ref{lijixian} and Lemma~\ref{xiangrong}.
\end{proof}

The algebra $\mathcal{A}$ is unital and simple, by density of the evaluation points. It has a unique tracial state $\tau$, since the induced maps on tracial state spaces contract norm distances by a factor of $1/3$. The six-term exact sequence gives $K_0(A_n)\cong\mathbb Z$ and $K_1(A_n)\cong\mathbb Z/2\mathbb Z$, with connecting maps given by multiplication by $3$ and the identity, respectively.
Consequently,
\[
\bigl(K_0(\mathcal{A}),K_0(\mathcal{A})^+,[1_\mathcal{A}]\bigr)
\cong
\bigl(\mathbb Z[1/3],\mathbb Z[1/3]_+,1\bigr),
\qquad
K_1(\mathcal{A})\cong\mathbb Z/2\mathbb Z.
\]

The building blocks have decomposition rank at most one \cite{W}, so $\mathcal{A}$ has finite nuclear dimension. Thus $\mathcal{A}$ is $\mathcal Z$-stable by \cite[Theorem~A]{C}. Since $\tau_*(K_0(\mathcal{A}))=\mathbb Z[1/3]$ is dense in $\mathbb R$, \cite[Theorem~6.7 and Corollary~7.3]{R5} imply that $\mathcal{A} $ has real rank zero and stable rank one.
\subsection{Nonuniqueness of AD Cartan subalgebras}\

We recall the Weyl groupoid facts needed below. Let $(A,B)$ be a
Cartan pair and put $X=\operatorname{Spec}(B)$. For $v\in N_A(B)$, define
\[
\operatorname{dom}(v)=\{\xi\in X:(v^*v)(\xi)>0\},
\qquad
\operatorname{ran}(v)=\{\xi\in X:(vv^*)(\xi)>0\}.
\]
The normalizer $v$ induces a homeomorphism
$\alpha_v\colon\operatorname{dom}(v)\to\operatorname{ran}(v)$, characterized by
\[
(v^*bv)(\xi)
=
b\bigl(\alpha_v(\xi)\bigr)(v^*v)(\xi),
\qquad b\in B,\quad \xi\in\operatorname{dom}(v).
\]
The germs of these partial homeomorphisms form the Weyl groupoid
$\mathcal G(A,B)$, whose unit space is $X$. Its orbit through
$\xi\in X$ is
\[
\mathcal G(A,B)\cdot\xi
=
\{r(\gamma):\gamma\in\mathcal G(A,B),\ s(\gamma)=\xi\}.
\]
An isomorphism of Cartan pairs induces an isomorphism of their Weyl
groupoids and hence maps Weyl orbits bijectively onto Weyl orbits;
see \cite{R3}.

For the system constructed above, write
\[
X_n=\operatorname{Spec}(B_n^\sigma),\qquad
X=\operatorname{Spec}(B^\sigma),\qquad
\mathcal G_n^\sigma=\mathcal G(A_n,B_n^\sigma).
\]
The restriction of $\varphi_n$ to $B_n^\sigma$ induces a
surjective map
\[ p_n\colon X_{n+1}\longrightarrow X_n,
\qquad
p_n(\xi)(b)=\xi\bigl(\varphi_n(b)\bigr)
\quad (b\in B_n^\sigma). \]
Thus
\[ X\cong\varprojlim(X_n,p_n). \]
We denote the coordinate projections by
$\pi_n\colon X\to X_n$, so that
$\pi_n=p_n\circ\pi_{n+1}$.

The corresponding statement for Weyl arrows requires an intermediate groupoid. By \cite[Proposition~5.4 and Remarks~5.5--5.6]{L2}, each connecting map gives a diagram
\[
\mathcal G_n^\sigma
\xleftarrow{\theta_n}
\mathcal H_n^\sigma
\xrightarrow{j_n}
\mathcal G_{n+1}^\sigma,
\]
where $j_n$ is an open embedding and $\theta_n$ is surjective, proper, and bijective on each source fibre. Since $\varphi_n$ is unital, $\mathcal H_n^\sigma$ has the full unit space $X_{n+1}$, and the restriction of $\theta_n$ to units is $p_n$. Li's construction realizes $\mathcal G(\mathcal{A},B^\sigma)$ as an increasing union of inverse limits formed from these intermediate groupoids. In particular, every limit Weyl arrow has, from some
stage onward, compatible finite stage coordinates $\gamma_m\in\mathcal G_m^\sigma$ whose source and range are $\pi_m(s(\gamma))$ and $\pi_m(r(\gamma))$, respectively. We use this coordinate description only to compare Weyl orbits; the compatibility maps for arrows are defined through the
intermediate groupoids.

When $p=2$, every component of $\Gamma_\sigma$ is a cycle. An order-$r$ circle is the realization of a cycle whose two bipartition classes each contain $r$ vertices. For $r=1$, this means the two-edge multigraph with adjacency matrix $(2)$. Let $h$ denote the natural height coordinate on such a circle.

\begin{lemma}\label{youxiang}
If $T_r\subseteq\operatorname{Spec}(B_n^\sigma)$ is an order-$r$ circle, then
\[
\#\bigl(
(\mathcal G(A_n,B_n^\sigma)\cdot\xi)\cap T_r
\bigr)
=
\begin{cases}
	2r, & h(\xi)\in(0,1),\\
	r,  & h(\xi)\in\{0,1\}.
\end{cases}
\]
\end{lemma}

\begin{proof}
Let $v$ be an element of the normalizer and $\xi\in\operatorname{dom}(v)$. Since $C([0,1])\subseteq Z(A_n)$, we have
\begin{align*}
	z(\xi)&=z(\alpha_v(\xi)),
\end{align*}
which holds for any $z\in C([0,1])$. So the partial homeomorphism induced by a normalizer preserves height. Interior matrix units connect all diagonal characters in a matrix fibre, and $T_r$ contains $2r$ characters over an interior height. At an endpoint, the fibre drops from $M_{2d_n}$ to $M_{d_n}$, and the circle contains $r$ characters. The endpoint matrix-unit lifts from Lemma~\ref{xiangrong} connect them.
\end{proof}

\begin{lemma}\label{houji}
After ordering the endpoint vertices as the old variable vertices followed by the constant vertices, the reduced matrices satisfy
\begin{equation*}
	R_{n+1}=R_n\oplus 2I_{2d_n}.
\end{equation*}
On spectra, every old circle has one variable successor on which the dual bonding map is a height-preserving homeomorphism. The two point-evaluation branches produce $2d_n$ new order-one circles, each of which is mapped constantly to a point at height $x_n$.
\end{lemma}

\begin{proof}
Order the endpoint vertices at both endpoints by placing first the $d_n$ vertices arising from the middle variable block and then the $2d_n$ vertices arising from evaluation at $x_n$.
	
The middle block inherits the two endpoint incidence maps from the preceding stage. It therefore contributes exactly the old reduced matrix $R_n$. For each of the $2d_n$ diagonal characters over $x_n$, the two identical point-evaluation blocks have the same new left and right endpoint vertices. They consequently give two parallel edges between the corresponding vertices. Altogether, these new components contribute the block $2I_{2d_n}$. Hence $R_{n+1}=R_n\oplus 2I_{2d_n}$. The assertions concerning the spectra follow directly from Gelfand duality.
\end{proof}

If $T$ is a circle component of $\operatorname{Spec}(B)$, define
\begin{equation*}
 \mu_B(T)=\sup_{\xi\in T}
 \#\bigl((\mathcal{G}(A,B)\cdot\xi)\cap T\bigr).
\end{equation*}

\begin{lemma}\label{bubian}
Suppose $\alpha:(C,E)\to(C',E')$ is an isomorphism of Cartan pairs, and let $\widehat\alpha:\operatorname{Spec}(E')\to\operatorname{Spec}(E)$ be the induced homeomorphism.  For every circle component $T\subseteq\operatorname{Spec}(E')$,
\[ \mu_{E'}(T)=\mu_E(\widehat\alpha(T)). \]
\end{lemma}

\begin{proof}
The Weyl construction turns $\alpha$ into an isomorphism $\mathcal G(C',E')\cong \mathcal G(C,E)$.  On unit spaces it restricts to $\widehat\alpha$ and maps orbits bijectively onto orbits.  A homeomorphism maps connected components onto connected components, proving the formula.
\end{proof}

\begin{lemma}\label{yuanzhou}
The non-singleton connected components of $\operatorname{Spec}(B^\sigma)$ are precisely the persistent circles obtained by following the variable successors of finite stage circles. In particular, every finite-stage order-$r$ circle gives rise to a connected component homeomorphic to $S^1$.
\end{lemma}

\begin{proof}
Write $X_n=\operatorname{Spec}(B_n^\sigma)$ and $X=\operatorname{Spec}(B^\sigma)$, and let $\pi_n\colon X\to X_n$ be the coordinate projections. By Lemma~\ref{houji}, every finite stage circle has a variable successor on which the dual bonding map is a homeomorphism. Each new point-evaluation circle is mapped constantly to a point.
	
Let $C_N\subseteq X_N$ be a finite stage circle, and let $C_n\subseteq X_n$, for $n\geq N$, be its successive variable successors. Their inverse limit
\[ K=\bigcap_{n\geq N}\pi_n^{-1}(C_n) \]
is homeomorphic to $C_N$, so $K$ is connected. To see that it is a connected component, take $y\in X\setminus K$. There is some $m\geq N$ such that $\pi_m(y)\notin C_m$. Since $C_m$ is a connected component of the finite disjoint union of circles $X_m$, it is clopen. Hence $U_m=\pi_m^{-1}(C_m)$ is clopen in $X$, contains $K$, and does not contain $y$. No connected subset of $X$ can contain both $K$ and $y$. Thus $K$ is a connected component.
	
Conversely, let $L$ be a connected component of $X$. For each $n$, the connected set $\pi_n(L)$ lies in a finite stage circle $D_n\subseteq X_n$. If the resulting component thread has point-evaluation successors infinitely often, then for every fixed $n$, a later constant bonding map makes $\pi_n(L)$ a singleton. Since the coordinate projections separate points, $L$ is a singleton. Otherwise the thread follows variable successors from some stage onward. Then $L$ lies in the persistent circle determined by that stage, and maximality of $L$ shows that it equals that circle. This proves the claim.
\end{proof}

\begin{proposition}\label{zhongshu}
For the persistent circle in Lemma~\ref{yuanzhou},
$\mu_{B^\sigma}(T_r)=2r$.
\end{proposition}

\begin{proof}
Suppose that $x,y\in T_r$ are joined by an arrow of $\mathcal G(\mathcal{A},B^\sigma)$. By the compatible finite stage description above, their coordinates are joined by Weyl arrows at every sufficiently large stage. These arrows preserve height, since every normalizer of $B_m^\sigma$ commutes with $C([0,1])$. The variable bonding maps preserve height as well. Hence $h(x)=h(y)$, so an orbit can meet $T_r$ only within a single height fibre. An interior height fibre contains $2r$ points, while an endpoint height fibre contains $r$ points.
	
Now fix $t\in(0,1)$ and $x,y\in T_r$ with $h(x)=h(y)=t$. At a stage $m$ at or after the birth of $T_r$, let $i,j$ be the diagonal indices of $x_m=\pi_m(x)$ and $y_m=\pi_m(y)$. Choose $\chi\in C_c((0,1))$ with $\chi(t)\neq0$, and set
\[ v(s)=\chi(s)U_m(s)e_{ji}U_m(s)^*. \]
Then $v\in N_{A_m}(B_m^\sigma)$, and its partial homeomorphism sends $x_m$ to $y_m$. The connecting maps preserve normalizers, so the image of $v$ in $\mathcal{A}$ normalizes $B^\sigma$. Along the variable-successor chains, its partial homeomorphism sends $x$ to $y$. Thus all $2r$ points of every interior height fibre belong to one Weyl orbit. By the definition of $\mu$, $\mu_{B^\sigma}(T_r)=2r.$
\end{proof}

At the first stage choose twists whose reduced matrices are
\begin{align*}
 R_{(1,1,1)}&=\begin{pmatrix}2&0&0\\0&2&0\\0&0&2\end{pmatrix},&
 R_{(2,1)}&=\begin{pmatrix}1&1&0\\1&1&0\\0&0&2\end{pmatrix},&
 R_{(3)}&=\begin{pmatrix}1&1&0\\1&0&1\\0&1&1\end{pmatrix}.
\end{align*}
Denote the corresponding limit Cartan subalgebras by $D$, $B^{(2,1)}$, and $B^{(3)}$, respectively. For a limit Cartan subalgebra $B$, set
\[
 \mathfrak M(B)=\{\mu_B(T):T\text{ is a circle component of }\operatorname{Spec}(B)\}.
\]
By Lemma~\ref{bubian}, isomorphic Cartan pairs have the same $\mathfrak M$.

\begin{proposition}\label{buweiyi}
The AD Cartan subalgebras above satisfy
\[
 \mathfrak M(D)=\{2\},\qquad
 \mathfrak M(B^{(2,1)})=\{2,4\},\qquad
 \mathfrak M(B^{(3)})=\{2,6\}.
\]
Consequently, the three Cartan pairs are pairwise nonisomorphic, and the three
Cartan subalgebras are pairwise nonconjugate in $\mathcal{A}$.
\end{proposition}

\begin{proof}
Apply Lemma~\ref{yuanzhou} and~Proposition~\ref{zhongshu}. The standard system contains only order-one circles. In the $(2,1)$ system, one persistent circle has order two, while the remaining circles born from untwisted summands have order one.  In the $(3)$ system, one persistent circle has order three and all other circles have order one. Therefore, the three distinct sets $\mathfrak M$ prove the assertion.
\end{proof}

\subsection{Infinite families of Cartan subalgebras in a fixed AD algebra}\

Recall that deleting finitely many initial terms does not change the inductive limit, so the tail system $(A_n,\varphi_n)_{n\geq N}$ still has inductive limit $\mathcal{A}$.

\begin{corollary}\label{wuxian}
There exists a unital AD algebra  containing infinitely many pairwise nonconjugate Cartan subalgebras, each of which is a $C^*$-diagonal.
\end{corollary}

\begin{proof}
Fix $r\geq 1$, and choose $N=N(r)$ such that $d_N\geq r$. Let $P_r$ be the permutation matrix of an $r$-cycle and set
\[ R^{[r]}=(I_r+P_r)\oplus 2I_{d_N-r}. \]
Every row and column of $R^{[r]}$ has sum $2$. Hence, by Theorem~\ref{keda}, it is the reduced matrix of a twist at stage $N$. Applying the preceding construction to the inductive system produces a Cartan subalgebra $B^{[r]}\subseteq \mathcal{A}$, which is a $C^*$-diagonal by Proposition~\ref{cunzai}.	
The block $I_r+P_r$ determines one order-$r$ circle, while all the remaining blocks determine order-one circles. Therefore, Lemma~\ref{houji}, Lemma~\ref{yuanzhou}, and Proposition~\ref{zhongshu} give
\[
\mathfrak M(B^{[r]})=\{2,2r\}.
\]
These sets are distinct for different $r$. Since conjugate Cartan subalgebras have the same invariant by Lemma~\ref{bubian}, the subalgebras $B^{[r]}$, $r\geq1$, are pairwise nonconjugate.
\end{proof}

In fact, the spectra of all Cartan subalgebras obtained from permutation twists at finite stages are mutually homeomorphic. To identify their common homeomorphism type, we use the dense amalgam introduced by \'Swi\k{a}tkowski; see \cite[Section~1]{S2}. In our case, we only need the following characterization for the dense amalgam of $S^1$.

\begin{lemma}[{\cite[Theorem~0.2]{S2}}]\label{dense}
Let $X$ be a compact metrizable space admitting a countably infinite family $\mathcal C$ of pairwise disjoint embedded circles such that:
\begin{enumerate}[label=\textup{(\roman*)}]
\item $\mathcal C$ is null: with respect to one, equivalently every, compatible metric on $X$, for every $\varepsilon>0$, only finitely many $T\in\mathcal C$ satisfy $\operatorname{diam}(T)\geq\varepsilon$;		
\item every $T\in\mathcal C$ has empty interior in $X$;	
\item $\bigcup_{T\in\mathcal C}T$ is dense in $X$;		
\item if $x,y\in X$ do not belong to the same member of $\mathcal C$, then there is a clopen $\mathcal C$-saturated subset $H\subseteq X$ such that $x\in H$ and $y\notin H$.
\end{enumerate}
Here, $H$ is $\mathcal C$-saturated if, for every
	$T\in\mathcal C$, either $T\subseteq H$ or
	$T\cap H=\varnothing$. Then
	\[
	X\cong\widetilde{\sqcup}(S^1).
	\]
\end{lemma}

In our setting, the family $\mathcal C$ consists of the persistent circle components from Lemma~\ref{yuanzhou}. Variable successors preserve these circles, whereas point-evaluation successors continually produce new order-one circles.

\begin{theorem}\label{michou}
Every Cartan subalgebra $B$ obtained from a permutation twist at a finite stage by the preceding inductive construction satisfies
\[ \operatorname{Spec}(B)\cong\widetilde{\sqcup}(S^1). \]
\end{theorem}

\begin{proof}
Suppose that the construction starts at stage $N$. Put $X_n=\operatorname{Spec}(B_n)$, $X=\operatorname{Spec}(B)\cong\varprojlim(X_n,p_n)$ and let $\pi_n\colon X\to X_n$ be the coordinate projections. For $m<k$, write $p_{m,k}=p_m\circ p_{m+1}\circ\cdots\circ p_{k-1}$. Let $\mathcal C$ be the family of persistent circles in $X$. 
By Lemmas~\ref{houji} and \ref{yuanzhou}, $\mathcal C$ is a countably infinite family of pairwise disjoint embedded circles. We verify the four conditions in Lemma~\ref{dense}.
\medskip
\noindent\textup{(i)}\enspace
For each $n\geq N$, choose a compatible metric $\delta_n$ on
$X_n$ such that $\delta_n\leq1$, and define
\[ d(x,y) = \sum_{n=N}^{\infty} 2^{-n}\delta_n\bigl(\pi_n(x),\pi_n(y)\bigr),
\qquad x,y\in X. \]
Suppose that $T\in\mathcal C$ first occurs at stage $m$. Then $\pi_n(T)$ is a singleton whenever $n<m$. Hence, for $x,y\in T$,
\[ d(x,y) \leq \sum_{n=m}^{\infty}2^{-n} = 2^{1-m}. \]
Thus $\operatorname{diam}(T)\leq2^{1-m}$. Given $\varepsilon>0$, choose $M$ such that $2^{1-M}<\varepsilon$. Every circle first occurring after stage $M$ then has diameter less than $\varepsilon$. Since only finitely many circles occur before stage $M$, the family $\mathcal C$ is null.

\medskip
\noindent\textup{(ii)}\enspace
Let $T\in\mathcal C$, let $z\in T$, and let $U$ be a neighbourhood of $z$. Choose $m$ and an open set $O\subseteq X_m$ such that $z\in\pi_m^{-1}(O)\subseteq U$, where $m$ is at least the stage at which $T$ first occurs. Since $\pi_m(T)$ is a circle and $O$ contains $\pi_m(z)$, the set $O\cap\pi_m(T)$ contains a nontrivial subarc. The heights of the points in this subarc therefore contain a nonempty open interval $J\subseteq(0,1)$.

Since $(x_n)$ is dense in $(0,1)$, choose $k\geq m$ such that $x_k\in J$. Choose a point $\zeta\in O\cap\pi_m(T)$ of height $x_k$, and let $\eta\in X_k$ be its lift along the variable successors. Thus $p_{m,k}(\eta)=\zeta$. 
At stage $k+1$, a point-evaluation successor produces a new order-one circle $E\subseteq X_{k+1}$ satisfying $p_k(E)=\{\eta\}$. Let $L\in\mathcal C$ be the persistent continuation of $E$. Then $\pi_m(L)=\{\zeta\}\subseteq O$, and hence $L\subseteq\pi_m^{-1}(O)\subseteq U$. Since $L\neq T$, every neighbourhood of $z$ contains a point not belonging to $T$. Therefore $T$ has empty interior.

\medskip
\noindent\textup{(iii)}\enspace
Let $U\subseteq X$ be a nonempty open set. Choose $m$ and a nonempty open set $O\subseteq X_m$ such that $\pi_m^{-1}(O)\subseteq U$. Fix $\xi_m\in O$. Inductively choose $\xi_{n+1}$ on the variable successor of the component containing $\xi_n$ so that
\[ p_n(\xi_{n+1})=\xi_n, \qquad n\geq m. \]
Together with the forced earlier coordinates, these points determine an element $\xi\in X$. Since its coordinates follow variable successors from stage $m$ onward, $\xi$ belongs to a persistent circle. Moreover, $\pi_m(\xi)=\xi_m\in O$, so $\xi\in U$. Consequently,
\[
\overline{\bigcup_{T\in\mathcal C}T}=X.
\]

\medskip
\noindent\textup{(iv)}\enspace
Let $x,y\in X$ be distinct points which do not belong to the same persistent circle. For each $n$, let $C_n(x)$ and $C_n(y)$ be the circle components of $X_n$ containing $\pi_n(x)$ and $\pi_n(y)$, respectively.

Suppose that $C_n(x)=C_n(y)$ for every $n$. If only finitely many of the corresponding bonding maps are constant, then the common component thread is eventually formed entirely by variable successors, so $x$ and $y$ belong to the same persistent circle. If infinitely many bonding maps are constant, then for every fixed $n$, a later constant bonding map forces $\pi_n(x)=\pi_n(y)$. Thus $x=y$. Both conclusions contradict our assumptions. Hence there is some $m$ such that $C_m(x)\neq C_m(y)$.

Set $H=\pi_m^{-1}\bigl(C_m(x)\bigr)$. Since $X_m$ is a finite disjoint union of circles, $C_m(x)$ is clopen in $X_m$. Therefore $H$ is clopen in $X$, with
$x\in H$ and $y\notin H$. Moreover, $H$ is $\mathcal C$-saturated. Indeed, if $T\in\mathcal C$, then $\pi_m(T)$ is connected and is therefore contained in a single circle component of $X_m$. Hence either
\[
T\subseteq H
\qquad\text{or}\qquad
T\cap H=\varnothing.
\]
Thus $\mathcal C$ satisfies the hypotheses of Lemma~\ref{dense}. Hence
\[ X=\operatorname{Spec}(B)\cong\widetilde{\sqcup}(S^1).\qedhere \]
\end{proof}

\begin{corollary}
There exists a unital AD algebra containing infinitely many mutually isomorphic but pairwise nonconjugate Cartan subalgebras, each of which is a $C^*$-diagonal.
\end{corollary}

\begin{proof}
This follows immediately from Corollary~\ref{wuxian} and Theorem~\ref{michou}.
\end{proof}

Recall that the AD algebra $\mathcal{A}$ constructed above is $\mathcal Z$-stable, where $\mathcal Z$ denotes the Jiang--Su algebra. Using this absorption property, we obtain the following result.

\begin{corollary}\label{butonggou}
The algebra $\mathcal{A}$ contains infinitely many Cartan subalgebras which are pairwise nonisomorphic as abelian $C^*$-algebras.
\end{corollary}

\begin{proof}
Let $B=B^{[1]}$ and put $X=\operatorname{Spec}(B)$. Since $X$ is the inverse limit of finite unions of circles, covering dimension does not increase under inverse limits and hence $\dim X\leq 1$. On the other hand, $X$ contains a persistent circle as a closed subspace. Therefore, $\dim X=1$.
	
By \cite[Proposition~5.1 and its proof]{LR2}, the Jiang--Su algebra $\mathcal Z$ contains Cartan subalgebras $D_j\subseteq\mathcal Z$, $j\geq1$, whose spectra $Y_j=\operatorname{Spec}(D_j)$ have finite and unbounded covering dimensions. Passing to a subsequence, we may assume that, with $e_j=\dim Y_j$, one has
\[ e_{j+1}>e_j+1
\qquad\text{for every }j\geq1. \]	
Choose an isomorphism $\Theta\colon \mathcal{A}\otimes\mathcal Z\longrightarrow \mathcal{A}$ and define
\[ H_j=\Theta(B\otimes D_j)\subseteq \mathcal{A}. \]
The tensor product of Cartan pairs is again a Cartan pair, so each $H_j$ is a Cartan subalgebra of $\mathcal{A}$. Moreover, 
\[ \operatorname{Spec}(H_j)\cong X\times Y_j. \]
For any $x\in X$, the inclusion
\[ Y_j\cong\{x\}\times Y_j\subseteq X\times Y_j \]
is a closed embedding. Together with the product inequality for covering dimension, this gives
\[ e_j \leq \dim(X\times Y_j) \leq \dim X+\dim Y_j = e_j+1. \]
The intervals $[e_j,e_j+1]$ are pairwise disjoint. Hence the spaces $\operatorname{Spec}(H_j)$ have pairwise distinct covering dimensions and are therefore pairwise nonhomeomorphic.
\end{proof}

\section*{AI statement}

During the preparation of this work, the author used GPT-5.6 Sol for the suggestion to identify the Cartan spectra with the dense amalgam of circles and for translation and language editing. The author has checked and substantially revised all arguments
and takes full responsibility for the final version.

\end{document}